\documentclass{amsart}

\usepackage{amsmath,amssymb,amsthm,amsfonts,mathrsfs,mathtools}
\usepackage[frame,cmtip,arrow,matrix,line,graph,curve]{xy}
\usepackage[mathscr]{eucal}
\usepackage{mathabx}
\usepackage[colorlinks,backref=page,citecolor=blue]{hyperref}
\usepackage{tikz}
\usepackage{epic,eepic}
\usepackage{yfonts}
\usepackage{enumerate}
\usepackage{comment}
\usepackage{enumitem}
\usepackage{stmaryrd}
\usepackage{tikz-cd}

\theoremstyle{definition}
\newtheorem{mainthm}{Theorem}
\newtheorem{mainthmm}{Theorem}

\newtheorem{theorem}{Theorem}[section]
\newtheorem{definition}[theorem]{Definition}

\newtheorem{lemma}[theorem]{Lemma}
\newtheorem{proposition}[theorem]{Proposition}
\newtheorem{corollary}[theorem]{Corollary}

\newtheorem{question}[theorem]{Question}
\newtheorem*{theorem*}{Theorem}

\theoremstyle{remark}
\newtheorem{remark}[theorem]{Remark}
\newtheorem{example}[theorem]{Example}

\def\sm{\setminus}
\def\PP{\mathbb{P}}

\def\RR{\mathbb{R}}
\def\CC{\mathbb{C}}
\def\ZZ{\mathbb{Z}}

\def\QQ{\mathbb{Q}}

\def\mcB{\mathcal{B}}

\def\mcF{\mathcal{F}}
\def\mcG{\mathcal{G}}
\def\mcH{\mathcal{H}}

\def\mcP{\mathcal{P}}
\def\mcQ{\mathcal{Q}}
\def\mcS{\mathcal{S}}
\def\umcQ{\underline\mcQ}
\def\umcS{\underline\mcS}

\def\eps{\varepsilon}
\def\uSigma{\underline\Sigma}
\def\uDelta{\underline\Delta}
\def\uPi{\underline\Pi}
\def\uX{\underline X}
\def\da{\mathord{\downarrow}}

\DeclareMathOperator{\cone}{cone}
\DeclareMathOperator{\Span}{span}

\DeclareMathOperator{\CH}{CH}
\DeclareMathOperator{\Nef}{Nef}
\DeclareMathOperator{\MW}{MW}
\DeclareMathOperator{\crk}{crk}
\DeclareMathOperator{\rk}{rk}
\DeclareMathOperator{\supp}{supp}
\DeclareMathOperator{\cl}{cl}
\DeclareMathOperator{\gaps}{gaps}
\DeclareMathOperator{\torstar}{star}
\DeclareMathOperator{\Elem}{Elem}
\DeclareMathOperator{\crem}{crem}

\numberwithin{equation}{section}

\begin{document}

\subjclass[2020]{Primary: 14C15, 14M25 Secondary: 52B40, 14C17.}
\keywords{Tautological classes of matroids, Bergman class, permutohedral variety, nef cone, extremal rays, stellahedral variety.}

\title{Extremality of Tautological Classes of Matroids in Nef Cones}

\author{Colin Stastny}
\address{Fachbereich Mathematik und Statistik, Universität Konstanz, Konstanz, Germany}
\email{Colin.Stastny@uni-konstanz.de}

\begin{abstract}
We study the extremality properties of rays generated by tautological Chern classes of matroids in the polyhedral nef cones of the permutohedral variety. Our main result provides necessary and sufficient conditions a matroid has to meet for its tautological Chern classes to generate extremal rays. Our approach yields a new and large family of extremal rays that interpolates between two previously known classes coming from rank functions and Bergman classes of matroids.
We further discuss how this theorem extends to the stellahedral variety through augmented tautological Chern classes.
\end{abstract}

\vspace{-1cm}
\maketitle

\vspace{-.5cm}

\section{Introduction} \label{sec:introduction}

Over the course of the last decade, it has become evident that many matroid invariants can be studied through the lens of algebraic geometry. In this context, one particular family of toric varieties appears repeatedly, namely the family of \textit{permutohedral varieties}. For a finite set $E$, the permutohedral variety $\uX_E$ is a smooth projective toric variety and, as such, its nef cones $\Nef^k(\uX_E)$ in the Chow ring $\CH^\bullet(\uX_E)$ are pointed and polyhedral. Hence, one can ask for a description of the extremal rays of these cones. This question has been asked in \cite{Huh_2016} and it is generally believed that describing all extremal rays is a difficult task. The high complexity of this problem already becomes clear in the case of the divisor nef cone $\Nef^1(\uX_E)$: 
It was shown in \cite[Cor.~2.5]{Batyrev_Blume_2011} that a divisor $D=\sum_{\emptyset\subsetneq S\subsetneq E}c_Sx_S\in\CH^1(\uX_E)$ is nef if and only if the function $c:2^E\to\QQ,\;S\mapsto~c_S$ with $c_\emptyset=c_E=0$ has the property that 
\begin{align} \label{eq:submodularity}
c_S+c_T\leq c_{S\cup T}+c_{S\cap T} \quad \text{holds for all $S,T\subset E$.}
\end{align}
Such a function $c$ is also called a \textit{submodular function} on $E$. The lineality space of the cone of submodular functions on $E$ consists precisely of those functions that satisfy equality in \eqref{eq:submodularity}, which, by definition, is the space of \textit{modular functions} on $E$. Thus, one has a linear isomorphism that identifies $\Nef^1(\uX_E)$ with the cone of submodular functions modulo modular functions. Describing all extremal submodular functions is a well-known and hard combinatorial problem with many far-reaching applications in various fields such as optimization, game theory and machine learning, see \cite{Balcan_Harvey_2012} and the references made therein. Techniques for generating extremal submodular functions on a small ground set, as well as general facts about the cone of submodular functions (or the closely related cone of supermodular functions) can be found in \cite{Studeny_2016}, \cite{Studeny_Bouckaert_Kocka_2000}, \cite{csirmaz_2024} and \cite{Loho_Padrol_Poullot_2025}.

At the other end, there is the nef cone of curves $\Nef^{n-1}(\uX_E)$, where $n=\dim(\uX_E)=|E|-1$. We will see in more detail in Example~\ref{exp:nef_1_and_n-1} that determining its extremal rays is equivalent to finding \textit{minimal balanced collections} on $E$. Again, this is another hard combinatorial problem and no explicit classification for minimal balanced collections is known.
Methods for generating minimal balanced collections have been studied in \cite{MGS_2023} and bounds on the number of minimal balanced collections were achieved in \cite{Bludov_Zuev_2025}.

Both of these instances suggest that characterizing all extremal rays of $\Nef^k(\uX_E)$ should also be hard for choices of $k$ other than $1$ and $n-1$. A natural approach to this problem is thus to first construct certain subsets of the set of all extremal rays. Two results of this type have been achieved in the past by invoking the theory of \textit{matroids}:
\begin{itemize}
    \item In \cite[Thm.~2.1.5]{Nguyen_1978}, it was shown that a matroid $M$ on $E$ is connected up to loops and coloops if and only if its rank function $\rk_M$ is an extremal submodular function on the set $E$. Under the correspondence mentioned above, this is equivalent to $\rk_M$ inducing an extremal ray of $\Nef^1(\uX_E)$.
    \item It was shown in \cite{Huh_2016} that the Bergman class of a loopless matroid $M$ on $E$ generates an extremal ray of $\Nef^{\crk(M)}(\uX_E)$.
\end{itemize}
We will see in Example~\ref{exp:c_1} and~\ref{exp:c_top} that the machinery of \textit{tautological Chern classes of matroids}, developed in \cite{Berget_Eur_Spink_Tseng_2023}, captures both of these constructions: Let $[\umcS_M]$ (resp. $[\umcQ_M]$) denote the tautological sub (resp. quotient) equivariant $K$-class of a matroid $M$ on $E$, let $[\umcS_M^\vee]$ (resp. $[\umcQ_M^\vee]$) be its dual $K$-class, and write $c_k(\cdot)$ for the $k$-th Chern class. Then, on the one hand, the extremal ray induced by the submodular function $\rk_M$ in $\Nef^1(\uX_E)$ is generated by $c_1(\umcQ_M)$ and, on the other hand, the Bergman class of $M$ is given by $c_{\crk(M)}(\umcQ_M)$, see \cite[Thm.~7.6]{Berget_Eur_Spink_Tseng_2023}. Since $c_k(\umcS^\vee_M)$ and $c_k(\umcQ_M)$ are always nef by Proposition~\ref{prop:Comb_Descr_MW_of_taut_classes}, this raises the question of whether tautological Chern classes of matroids could interpolate between the two known classes of extremal rays from above. In other words:
\begin{center}
    \textit{Under suitable assumptions on a matroid $M$ on $E$, are the rays generated by the\\
    tautological Chern classes $c_k(\umcS^\vee)$ and $c_k(\umcQ)$ extremal in $\Nef^k(\uX_E)$?}
\end{center}
As the first main result of this article, we will answer this question in the affirmative by showing that connectedness up to loops and coloops is a necessary and sufficient condition for the tautological Chern classes to be extremal:

\begin{mainthm} \label{thm:main_result_intro}
Let $M$ be a matroid on $E$, let $l$ be the number of loops and $l'$ be the number of coloops of $M$. Then, for any integers $k$ and $k'$ with $1\leq k<\crk(M)-l$ and $1\leq k'<\rk(M)-l'$, the following are equivalent:
\begin{enumerate}
    \item $M$ is connected up to loops and coloops, \label{eq:main_result_1_intro}
    \item $c_k(\umcQ_M)$ generates an extremal ray of $\Nef^k(\uX_E)$, \label{eq:main_result_2_intro}
    \item $c_{k'}(\umcS_M^\vee)$ generates an extremal ray of $\Nef^{k'}(\uX_E)$. \label{eq:main_result_3_intro}
\end{enumerate}
\end{mainthm}

Note that, since the first condition of the theorem is independent of $k$ and $k'$, connectedness up to loops and coloops implies that all appearing tautological sub and quotient Chern classes generate extremal rays in their respective nef cones. Thus, if one of the above tautological Chern classes generates an extremal ray, then all of them do.

Another closely related toric variety that is used to study matroidal invariants through intersection theory is the \textit{stellahedral variety} $X_E$. In \cite{Eur_Huh_Larson_2023}, the authors constructed the augmented Bergman class and augmented tautological sub (resp. quotient) $K$-class $[\mcS_M]$ (resp. $[\mcQ_M]$) of a matroid $M$, which extend the corresponding constructions made on the permutohedral variety to the stellahedral variety. These $K$-classes come with induced augmented tautological Chern classes which naturally leads to the question of whether Theorem~\ref{thm:main_result_intro} can be extended to obtain extremal rays of the nef cone of the stellahedral variety. While augmented tautological sub Chern classes are rarely even nef or anti-nef by \cite[Prop.~5.9]{Eur_Huh_Larson_2023}, we will see in Chapter~\ref{sec:main_result_stella} that the equivalences from Theorem~\ref{thm:main_result_intro} can indeed be related to the extremality of augmented tautological quotient Chern classes:

\begin{mainthm} \label{thm:main_result_stella_intro}
If $M$ is loopless, the three conditions from Theorem~\ref{thm:main_result_intro} are equivalent to
\begin{enumerate} [start = 4]
    \item $c_k(\mcQ_M)$ generates an extremal ray of $\Nef^k(X_E)$.
\end{enumerate}
\end{mainthm}

We will further see that, for matroids that contain loops, the class $c_k(\mcQ_M)$ cannot be extremal. Our approach for proving this theorem will additionally allow us to deduce the extremality of the ray generated by the augmented Bergman class of any loopless matroid $M$ in $\Nef^{\crk(M)}(X_E)$.

\subsection*{Acknowledgements}
I would like to thank my advisor, Mateusz~Micha\l{}ek, for supervising this project, as well as for his incredible support during my Master’s program. Further, I want to thank Matt~Larson who hinted me towards tautological Chern classes of matroids in the first place.

\section{Preliminaries} \label{sec:preliminaries}

\subsection*{The permutohedral variety}
In this paper, we expect the reader to be familiar with the basic theory of toric varieties. For a comprehensive treatment of toric varieties and notational questions, consult \cite{Cox_Little_Schenck_2011}. A compact introduction to toric varieties can be found in \cite[Ch.~8]{Michalek_Sturmfels_2021}.

Let $E$ be a finite set, $n=|E|-1$ and write $[n]\coloneqq\{0,\dots,n\}$. Let further $(\mathbf{e}_i)_{i\in E}$ be the canonical basis of the lattice $\ZZ^E$. For $S\subset E$, we put
\begin{align*}
\mathbf{e}_S\coloneqq\sum_{i\in S}\mathbf{e}_i\in\ZZ^E.
\end{align*}
Let us fix any labeling $\lambda:[n]\to E$ and, by abuse of notation, write $\mathbf{e}_i$ for $\mathbf{e}_{\lambda(i)}$. Having this, the \textit{permutohedron} is the lattice polytope in $\RR^E$ given by
\begin{align*}
    \uPi_E'\coloneqq\operatorname{conv}\bigl( \sigma(0)\mathbf{e}_{0}+\dots+\sigma(n)\mathbf{e}_{n} \,\big|\, \sigma\in\mathfrak S_{[n]}\bigr).
\end{align*}
Note that the resulting polytope $\uPi_E'$ is independent of the chosen labeling $\lambda$.
Alternatively, the permutohedron can be expressed as the Minkowski sum
\begin{align} \label{eq:Perm_as_hypersimplices}
    \uPi_E'=\sum_{i=1}^n\underline\Delta_{E,i},\quad\text{where }\underline\Delta_{E,i}\coloneqq\{x\in[0,1]^E\mid x_0+\cdots+x_n=i\}
\end{align}
is the $i$-th hypersimplex. It is contained in the hyperplane $H\subset\RR^E$ consisting of those points whose coordinates sum to $\sum_{i=0}^ni=\frac{n(n+1)}{2}$. Hence, one can consider the translated polytope $\uPi_E\coloneqq\uPi_E'-\frac{n(n+1)}{2}\mathbf{e}_n\subset \RR^E$, which is also called permutohedron, as a full-dimensional lattice polytope in $\ZZ^E\cap\Span(\mathbf{e}_E)^\perp=\Span_\ZZ(\mathbf{e}_0-\mathbf{e}_n,\dots,\mathbf{e}_{n-1}-\mathbf{e}_n)\cong\ZZ^n$. Here, the orthogonal complement is taken with respect to the standard inner product.
The lattice polytope $\uPi_E$ induces a rational normal fan $\uSigma_E\coloneqq\Sigma_{\uPi_E}$ in the dual lattice $\ZZ^E/\Span_\ZZ(\mathbf{e}_E)$ which is called the \textit{permutohedral fan}. The geometry of this fan can be understood through \textit{flags}, where, by a flag of length $d$ on $E$, we mean a strictly increasing chain $\bigl(F_1\subsetneq\cdots\subsetneq F_d\bigr)$ consisting of subsets of $E$.

\begin{proposition} \label{prop:geometry_of_perm_fan}
    The set $\uSigma_E(k)$ of $k$-codimensional cones of $\uSigma_E$ is in bijection with the set of flags of length $n-k$ consisting of non-empty proper subsets of $E$ via
    \begin{align*}
        \mcF=\bigl(F_1\subsetneq\cdots\subsetneq F_{n-k}\bigr) \mapsto \sigma_\mcF\coloneqq\cone(\overline{\mathbf{e}_{F_1}},\dots,\overline{\mathbf{e}_{F_{n-k}}}).
    \end{align*}
    In particular, the rays of $\uSigma_E$ bijectively correspond to non-empty proper subsets of $E$. The maximal cones of $\uSigma_E$ are in bijection with permutations on $E$, since every full flag on $E$ can, after relabeling, uniquely be written as $\sigma([0])\subsetneq\sigma([1])\subsetneq\cdots\subsetneq\sigma([n])$ for some $\sigma\in\mathfrak S_{[n]}$. Moreover, we have $\sigma_\mcF\prec\sigma_\mcG$, i.e., $\sigma_\mcF$ is a face of $\sigma_\mcG$, if and only if $\mcG$ contains every set from $\mcF$. In this case, we also write $\mcF\prec\mcG$.
\end{proposition}

Since the permutohedral fan $\uSigma_E$ is complete and smooth, it gives rise to a complete smooth toric variety $\uX_E\coloneqq X_{\uSigma_E}$ with dense open torus $(\CC^*)^E/\CC^*$, where $\CC^*$ acts diagonally on $(\CC^*)^E$. The variety $\uX_E$ is called the \textit{permutohedral variety}. It will be the central geometric object that we study in this article. 

\subsection*{The Chow ring of the permutohedral variety}
An important algebraic invariant that formalizes intersection theory on a smooth quasi-projective variety $X$ is its \textit{Chow ring}. If $X$ is equi-dimensional, this ring is graded by codimension and we denote it by $\CH^\bullet(X)$. Here, we will always work with coefficients from $\QQ$, i.e., cycles are $\QQ$-linear combinations of classes of subvarieties. An introduction to Chow rings of schemes can be found in \cite[Ch.~1.2]{Eisenbud_Harris_2016}.
For toric varieties, there are multiple combinatorial descriptions known for their Chow rings: If $\Sigma$ is a $n$-dimensional complete smooth fan, the Chow ring $\CH^\bullet(X_\Sigma)$ is, as a graded ring, isomorphic to the following:
\begin{enumerate}
    \item The Chow ring of the fan $\Sigma$, defined by
    \begin{align*}
    \CH^\bullet(\Sigma) \coloneqq \QQ[x_\rho\mid\rho\in\Sigma(n-1)]/(\mathscr I_\Sigma+\mathscr J_\Sigma),
    \end{align*}
    where $\mathscr I_\Sigma$ is the Stanley-Reisner ideal of the simplicial complex defined by $\Sigma$ and $\mathscr J_\Sigma$ is generated by certain linear forms. For reference, see \cite[Ch.~12.4]{Cox_Little_Schenck_2011}.
    
    \item The quotient $\CH_T^\bullet(\uX_E)/\ker(\varphi)$, where $\CH_T^\bullet(\uX_E)$ is the $T$-equivariant Chow ring of $\uX_E$ defined as in \cite{Edidin_Graham_1997} that comes with a surjection $\varphi:\CH_T^\bullet(\uX_E)\to\CH^\bullet(\uX_E)$. By \cite{Payne_2006}, this quotient can be identified with the ring of piecewise polynomial functions on $\Sigma$ modulo the ideal generated by global polynomials that vanish at the origin. If the cones from $\Sigma$ are pointed, the isomorphism in degree one sends a piecewise linear function $l$ on $\Sigma$ to $\sum_{\rho\in\Sigma(n-1)}l(\mathbf{e}_\rho)x_\rho$, where $\mathbf{e}_\rho$ is the primitive ray generator of $\rho$.

    \item The ring of Minkowski weights $\MW^\bullet(\Sigma)\coloneqq\bigoplus_{k=0}^n\MW^k(\Sigma)$. Here, the space of (rational) Minkowski weights of codimension $k$ is the set of functions $\Delta:\Sigma(k)\to\QQ$ that, for every $\tau\in\Sigma(k+1)$, fulfill the following \textit{balancing condition}
    \begin{align*}
        \sum_{\sigma\succ\tau}\Delta(\sigma)\mathbf{e}_{\sigma\sm\tau} \in \Span_\ZZ(\tau).
    \end{align*}
    Here, $\mathbf{e}_{\sigma\sm\tau}$ is the primitive generator of the unique ray in $\sigma$ that does not lie in $\tau$. The multiplication on $\MW^\bullet(\Sigma)$ is given by stable intersection of tropical cycles which can be computed using the fan displacement rule from \cite{Fulton_Sturmfels_1997}. Explicitly, the isomorphism sends $\xi\in\CH^k(X_\Sigma)$ to $\xi\cap\Delta_\Sigma$, where
    \begin{align*}
        (\xi\cap\Delta_\Sigma)(\sigma)\coloneqq\int_{X_\Sigma}\xi\cdot[V(\sigma)].
    \end{align*}
    We denote the Chow class corresponding to a Minkowski weight $\Delta$ by $\Delta\cap[X_\Sigma]$.
\end{enumerate}
Consult \cite{Ardila_2024} for an in depth discussion of these three isomorphisms. The vector spaces $\MW^k(\Sigma)$ are defined the same as above for a non-complete smooth fan $\Sigma$.
However, in this case, $\MW^\bullet(\Sigma)$ might not admit a ring structure. Specializing to the permutohedral fan and using Proposition~\ref{prop:geometry_of_perm_fan}, this yields the following explicit descriptions of the Chow ring of the permutohedral variety:
\begin{proposition} \label{prop:Chow_of_perm}
    \begin{enumerate}
        \item There holds $\CH^\bullet(\uX_E)\cong\QQ[x_F]_{\emptyset\subsetneq F\subsetneq E}/(\mathscr I+\mathscr J)$, where
        \begin{align*}
            \mathscr I=\langle x_{F_1}x_{F_2}\mid F_1\not\subset F_2\text{ and }F_2\not\subset F_1\rangle
            \quad\text{and}\quad
            \mathscr J=\left\langle \sum_{F\ni i}x_F-\sum_{F\ni j}x_F \,\middle|\, i,j\in E\right\rangle.
        \end{align*}
        If $\mcF=\bigl(F_1\subsetneq\cdots\subsetneq F_k\bigr)$ is a flag of non-empty proper subsets of $E$, the Chow class of its induced codimension $k$ toric subvariety $V(\sigma_\mcF)\subset\uX_E$ corresponds to
        \[x_\mcF\coloneqq\prod_{i=1}^kx_{F_i}\in\CH^k(\uX_E).\]
        
        \item One can identify the $T$-equivariant Chow ring $\CH_T^\bullet(\uX_E)$ with 
        \begin{align*}
            \left\{f\in\prod_{\sigma\in\mathfrak S_E}\QQ[t_i\mid i\in E]\,\middle|\, \begin{array}{c}f_\sigma-f_{\sigma'}\equiv0\mod t_{\sigma(i)}-t_{\sigma(i+1)}\\ \text{whenever }\sigma'=\sigma\circ(i,i+1)\text{ for some }i\in E\end{array}\right\}
        \end{align*}
        and one obtains $\CH^\bullet(\uX_E)$ by quotienting with the ideal of global polynomials on $\RR^E$ that vanish at the origin.
        
        \item We have $\CH^\bullet(\uX_E)\cong\MW^\bullet(\uSigma_E)$ and, by Proposition~\ref{prop:geometry_of_perm_fan},  one can interpret a Minkowski weight $\Delta$ on $\uSigma_E$ of codimension $k$ as a function on flags of non-empty proper subsets of $E$ of length $n-k$. For convenience, we will therefore sometimes write $\Delta(\mcF)$ for $\Delta(\sigma_\mcF)$.
    \end{enumerate}
\end{proposition}

Besides the $\mathfrak S_E$-symmetry of $\uSigma_E$ corresponding to relabeling elements of $E$, the permutohedral fan $\uSigma_E$ inherits another symmetry: Recall that the toric variety corresponding to the $n$-simplex is $\PP^n$ which comes with the \textit{Cremona involution}
\begin{align*}
    \crem:\PP^n\dashrightarrow\PP^n,\quad [x_0:\cdots:x_n]\mapsto[x_0^{-1}:\cdots:x_n^{-1}].
\end{align*}
By \eqref{eq:Perm_as_hypersimplices}, the $n$-simplex $\underline\Delta_{E,1}$ is a Minkowski-summand of the permutohedron and hence the fan corresponding to $\PP^n$ coarsens the permutohedral fan. Thus, one has a birational toric morphism $\pi:\uX_E\to\PP^n$. Moreover, the involution $x\mapsto -x$ on the fan $\uSigma_E$ induces a toric morphism $\uX_E\to\uX_E$, which we will again denote by $\crem$. This is justified since the following diagram commutes:
\[\begin{tikzcd}
	{\uX_E} && {\uX_E} \\
	{\PP^n} && {\PP^n}
	\arrow["\crem", from=1-1, to=1-3]
	\arrow["\pi", from=1-1, to=2-1]
	\arrow["\pi", from=1-3, to=2-3]
	\arrow["\crem", dashed, from=2-1, to=2-3]
\end{tikzcd}\]

\noindent Since $\crem$ is an involution on $\uX_E$, the induced pullback $\crem^*$ and pushforward $\crem_*$ maps on Chow classes are equal and we will again just write them as $\crem$. As $-\overline{\mathbf{e}_F}=\overline{\mathbf{e}_{E\sm F}}$, there holds $\crem(x_S)=x_{E\sm S}$ in the notation from Proposition~\ref{prop:Chow_of_perm}.

\subsection*{Nef cones}
Let $X$ be a $n$-dimensional complete smooth variety and let $0\leq k\leq n$. For subvarieties $Y_1,\dots,Y_m$ of $X$ of dimension $k$, the cycle $\sum_{i=1}^mc_i[Y_i]$ is called \textit{effective} if $c_i\geq0$ for all $i=1,\dots,m$. Further, a Chow class is called \textit{effective}, if it is the class of an effective cycle.
A $k$-codimensional Chow class is called \textit{nef} (short for \textit{numerically effective}), if it pairs non-negatively with every effective $k$-dimensional Chow class under the degree map. The set of $k$-codimensional nef Chow classes forms a cone, the \textit{nef cone} $\Nef^k(X)\subset\CH^k(X)$ of $X$.\\
Whenever $\Sigma$ is complete and smooth, one can prove that every effective cycle of $X_\Sigma$ is rationally equivalent to a torus-invariant effective cycle, see \cite{Fulton_MacPherson_Sottile_Sturmfels_1995}. In particular, $\xi\in\CH^k(X_\Sigma)$ is nef if and only if
\begin{align}\label{eq:nef_for_toric}
    (\xi\cap\Delta_\Sigma)(\sigma)=\int_{X_\Sigma}\xi\cdot[V(\sigma)]\geq0\quad\text{for all }\sigma\in\Sigma(k).
\end{align}
Thus, the nef cones of $X_\Sigma$ are polyhedral and under the isomorphism $\CH^k(X_\Sigma)\cong\MW^k(X_E)$, the nef cone $\Nef^k(X_\Sigma)$ corresponds to the cone of non-negative $k$-codimensional Minkowski weights on $\Sigma$. We will denote this cone (also for possibly non-complete $\Sigma$) by $\MW^k_{\geq0}(\Sigma)$.

Having this, we can characterize the extremal rays of $\Nef^k(X_\Sigma)$ via Minkowski weights: To this end, we define $\Sigma_\Delta$ for $\Delta\in\MW^k(\Sigma)$ to be the subfan of $\Sigma$ whose maximal cones are precisely those from $\supp(\Delta) = \{\sigma\in\Sigma(k)\mid\Delta(\sigma)\neq0\}$.

\begin{proposition} \label{prop:char_nef_rays}
    For a non-zero Chow class $\xi\in\CH^k(X_\Sigma)$ with induced Minkowski weight $\Delta=\xi\cap\Delta_{\Sigma}\in\MW^k(\Sigma)$, the following are equivalent:
    \begin{enumerate}
        \item $\xi$ generates an extremal ray of $\Nef^k(X_\Sigma)$,\label{eq:char_nef_rays_1}
        \item $\Delta$ generates an extremal ray of $\MW^k_{\geq0}(\Sigma)$,\label{eq:char_nef_rays_2}
        \item $\Sigma_\Delta$ is inclusion-minimal in $\{\Sigma_{\Delta'}\mid \Delta'\in\MW^k_{\geq0}(\Sigma)\text{ non-zero}\}$, \label{eq:char_nef_rays_3}
        \item $\MW^k_{\geq0}(\Sigma_\Delta)=\QQ_{\geq0}\Delta$. \label{eq:char_nef_rays_4}
    \end{enumerate}
\begin{proof}
    The equivalence of \eqref{eq:char_nef_rays_1} and \eqref{eq:char_nef_rays_2} follows from the above discussion.\\
    '\eqref{eq:char_nef_rays_2}$\Rightarrow$\eqref{eq:char_nef_rays_3}': If there is a non-zero $\Delta'\in\MW^k_{\geq0}(\Sigma_\Delta)$ with $\Sigma_{\Delta'}\subsetneq\Sigma_\Delta$, then we can choose $\eps\in\QQ_{>0}$ small such that $\Delta-\eps\Delta'\in\MW^k_{\geq0}(\Sigma_\Delta)$. The decomposition $\Delta=(\Delta-\eps\Delta')+\eps\Delta'$ shows that $\Delta$ does not generate an extremal ray in $\MW^k_{\geq0}(\Sigma)$.\\
    '\eqref{eq:char_nef_rays_3}$\Rightarrow$\eqref{eq:char_nef_rays_4}': Let $\Delta'\in\MW^k_{\geq0}(\Sigma_\Delta)$ be nonzero. By condition \eqref{eq:char_nef_rays_3}, we then have $\Sigma_{\Delta'}=\Sigma_\Delta$. Let $\sigma\in\Sigma_\Delta(0)$ such that $\frac{\Delta'(\sigma)}{\Delta(\sigma)}$ is minimal and note that $\Delta'-\frac{\Delta'(\sigma)}{\Delta(\sigma)}\Delta$ is non-negative with support strictly contained in the support of $\Delta$. Therefore, we have $\Delta'=\frac{\Delta(\sigma)}{\Delta'(\sigma)}\Delta$ and hence $\Delta'\in\QQ_{\geq0}\Delta$.\\
    '\eqref{eq:char_nef_rays_4}$\Rightarrow$\eqref{eq:char_nef_rays_2}': Let $\Delta'$ be a summand of $\Delta$ in $\MW^k_{\geq0}(\Sigma)$. Since the difference $\Delta-\Delta'$ is non-negative, we find $\Sigma_{\Delta'}\subset\Sigma_\Delta$ and thus $\Delta'\in\MW^k_{\geq0}(\Sigma_\Delta)$. By condition \eqref{eq:char_nef_rays_4}, the weight $\Delta'$ is a non-negative multiple of $\Delta$.
\end{proof}
\end{proposition}

\begin{example} \label{exp:nef_1_and_n-1}
We can use this characterization to describe the nef cone of curves $\Nef^{n-1}(\uX_E)$: Recall that, by Propositions~\ref{prop:geometry_of_perm_fan} and~\ref{prop:Chow_of_perm}, a Minkowski weight $\Delta$ from $\MW^{n-1}(\uSigma_E)$ can be regarded as a function from $2^E\sm\{\emptyset,E\}$ to $\QQ$ that additionally fulfills
\begin{align*}
    \Bigg(\sum_{\substack{\emptyset\subsetneq S\subsetneq E\\i\in S}}\Delta(S)\Bigg)_{i\in E}\in\Span(\mathbf{e}_E).
\end{align*}
Let $L_E$ denote the affine subspace of the space of $\MW^{n-1}(\uSigma_E)$ consisting of those $\Delta$ for which the above vector is $\mathbf{e}_E$, i.e., for which all entries are one. Supports of elements from the polytope $W_E\coloneqq\MW^{n-1}_{\geq0}(\uSigma_E)\cap L_E$ are also regarded as \textit{balanced collections} on $E$, which generalize partitions of sets: Any partition $\mcP$ of $E$ consisting of non-empty proper subsets can be realized as a balanced collection, since it is the support of
\begin{align*}
    \Delta(S)=\begin{cases}
        1 & \text{if }S\in\mcP\\
        0 & \text{otherwise}
    \end{cases}.
\end{align*}
The extremal rays of $\MW^{n-1}_{\geq0}(\uSigma_E)$ are generated by the vertices of $W_E$, which, by Proposition~\ref{prop:char_nef_rays}, are in bijection with balanced collections on $E$ that are minimal with respect to inclusion. As mentioned in the introduction, these so called \textit{minimal balanced collections} have been studied before from a purely combinatorial point of view and independent from this geometric perspective, see \cite{MGS_2023}.
\end{example}

\section{Statement and proof of the main result} \label{sec:main_result}

\subsection*{Bergman classes and tautological classes of matroids}
For a thorough introduction to notations and results related to the theory of matroids, we refer the reader to \cite{Oxley_2006} and \cite{Katz_2016}. We will briefly recall the construction of Bergman classes and tautological Chern classes of matroids made in \cite{Ardila_Klivans_2005} and \cite{Berget_Eur_Spink_Tseng_2023} respectively:\\
For any matroid $M$ on $E$, one can prove that all its maximal flags of flats have the same length, namely $\rk(M)+1$. Therefore, if $M$ is loopless, one can define the \textit{Bergman fan} $\uSigma_M$ of $M$ as the $\rk(M)-1$-dimensional subfan of $\uSigma_E$ whose maximal cones are given by those $\sigma_\mcF$, where $\mcF$ consists only of flats of $M$ that are non-empty and proper subsets of $E$. For any, not necessarily loopless, matroid $M$ on $E$, this construction gives rise to the weight function
\begin{align*}
    \uDelta_M:\uSigma_E(\crk(M))\to\QQ, \quad
    \sigma_\mcF\mapsto\begin{cases}
        1 &\text{if $M$ is loopless and }\sigma_\mcF\in\uSigma_M\\
        0 &\text{otherwise}
    \end{cases}.
\end{align*}
By the flat partition property of matroids, this defines a balanced Minkowski weight and hence it induces the Chow class $\uDelta_M\cap[\uX_E]$, called the \textit{Bergman class} of $M$.

Let $M$ be a matroid on $E$ and fix a linear order $\preceq$ on $E$. Any $\sigma\in\mathfrak S_E$ then induces a linear order $\preceq_\sigma$ on $E$ via $\sigma(e)\preceq_\sigma\sigma(f):\Leftrightarrow e\preceq f$, which defines a lexicographic order on cardinality $\rk(M)$ subsets of $E$ by $A\leq_\sigma B:\Leftrightarrow \min_{\preceq_\sigma}(A\sm B)\preceq_\sigma\min_{\preceq_\sigma}(B\sm A)$. Having this, the \textit{lex-first basis of $M$ with respect to $\sigma$} is given by $B_\sigma(M)\coloneqq\min_{\leq_\sigma}\mcB(M)$, where $\mcB(M)$ is the set of bases of $M$. The \textit{tautological sub (resp. quotient) equivariant Chern classes} $c_k^T(\umcS_M^\vee)$ (resp. $c_k^T(\umcQ_M)$) of $M$ is the equivariant Chow class associated (in the sense of Proposition~\ref{prop:Chow_of_perm}) to the piecewise polynomial function
\begin{align*} 
    c_k^T(\umcS_M^\vee)_\sigma
    \coloneqq\Elem_k(B_\sigma(M)) \quad\text{or resp.}\quad c_k^T(\umcQ_M)_\sigma
    \coloneqq(-1)^k\Elem_k(E\sm B_\sigma(M)).
\end{align*}
Here, $\Elem_k(S)$ denotes the $k$-th elementary symmetric polynomial in those variables that are indexed by $S$:
\begin{align*}
    \Elem_k(S)\coloneqq\sum_{T\subset S,|T|=k}\prod_{i\in T}t_i\in\QQ[t_i\mid i\in E].
\end{align*}
Note that $c_k^T(\umcS_M^\vee)$ and $c_k^T(\umcQ_M)$ are zero for $k>\rk(M)$ and $k>\crk(M)$ respectively.
The image of $c_k^T(\umcS_M^\vee)$ (resp. $c_k^T(\umcQ_M)$) under the surjection $\CH_T^\bullet(\uX_E)\to\CH^\bullet(\uX_E)$ is denoted by $c_k(\umcS_M^\vee)$ (resp. $c_k(\umcQ_M)$) and is called the \textit{tautological sub (resp. quotient) Chern class of $M$}. By \cite[Prop.~5.11]{Berget_Eur_Spink_Tseng_2023}, one has the following duality property:
\begin{align} \label{eq:duality_property}
    c_k(\umcS_M^\vee)=\crem c_k(\umcQ_{M^*}).
\end{align}
The proof of our main theorem will heavily rely on the explicit description of the Minkowski weights associated to tautological classes of matroids given in \cite[Prop.~7.4]{Berget_Eur_Spink_Tseng_2023}:

\begin{proposition} \label{prop:Comb_Descr_MW_of_taut_classes}
    Let $M$ be a matroid on $E$ and $k\in[n]$. Let further $\mcF=\bigl(F_1\subsetneq\cdots\subsetneq F_{n-k}\bigr)$ be a flag consisting of non-empty proper subsets of $E$ and put $F_0=\emptyset$ and $F_{n-k+1}=E$. There holds
    \begin{align*}
        \bigl(c_k(\umcS_M^\vee)\cap\Delta_{\uSigma_E}\bigr)(\sigma_\mcF)
        &=\begin{cases}
            1 & \begin{aligned}
            &\text{if, for $i\in[n-k]$, exactly $\rk(M)-k$ minors $M|F_{i+1}/F_i$}\\ &\text{are coloops and the rest are uniform of corank one}
            \end{aligned}\\
            0 & \text{otherwise}
        \end{cases}
    \end{align*}
    and
    \begin{align*}
        \bigl(c_k(\umcQ_M)\cap\Delta_{\uSigma_E}\bigr)(\sigma_\mcF)
        &=\begin{cases}
            1 & \begin{aligned}
            &\text{if, for $i\in[n-k]$, exactly $\crk(M)-k$ minors $M|F_{i+1}/F_i$}\\ &\text{are loops and the rest are uniform of rank one}
            \end{aligned}\\
            0 & \text{otherwise}
        \end{cases}.
    \end{align*}
    In particular, $c_k(\umcS_M^\vee)$ and $c_k(\umcQ_M)$ are nef.
\end{proposition}

We will now discuss the two extreme cases of tautological Chern classes of matroids and relate them to the two classes of extremal nef rays from the introduction:

\begin{example} \label{exp:c_1}
    Note that 
    \begin{align*}
        c_1^T(\umcS_M^\vee)_\sigma(\mathbf{e}_S) = \Elem_1(B_\sigma(M))(\mathbf{e}_S)=|B_\sigma(M)\cap S|,
    \end{align*}
    where $\sigma$ is any permutation that sends the $|S|$ smallest (with respect to $\preceq$) elements of $E$ to $S$. By the construction of lex-first bases, it follows that for such $\sigma$, the basis $B_\sigma(M)$ intersects $S$ maximally among all bases of $M$ and therefore $|B_\sigma(M)\cap S|=\rk_M(S)$. In order to compute $c_1(\umcS_M^\vee)$ in the Chow ring of the fan $\uSigma_E$, we have to subtract the global polynomial $\frac{\rk(M)}{n+1}\sum_{i\in E}t_i$ from $c_1^T(\umcS_M^\vee)$ to obtain a piecewise linear function on $\uSigma_E$ which shows that
    \begin{align*}
        c_1(\umcS_M^\vee)= \sum_{\emptyset\subsetneq S\subsetneq E}\left(\rk_M(S)-\frac{\rk(M)}{n+1}|S|\right)x_S.
    \end{align*}
    As the piecewise linear function $c_1^T(\umcQ_M)$ differs from $c_1^T(\umcS_M^\vee)$ by the global linear function $\sum_{i\in E}t_i$, the same description holds true for the first tautological quotient Chern class, i.e., we have $c_1(\umcQ_M)=c_1(\umcS_M^\vee)$.
    In \cite[Thm.~2.1.5]{Nguyen_1978}, it was shown that a matroid $M$ on $E$ is connected up to loops and coloops if and only if its rank function $\rk_M$ generates an extremal ray of the cone of submodular functions modulo modular functions. Using the identification from the introduction and the fact that $2^E\to\QQ,S\mapsto\frac{\rk(M)}{n+1}|S|$ is modular, this precisely means that $c_1(\umcQ_M)$ generates an extremal ray of $\Nef^1(\uX_E)$.
\end{example}

\begin{example}\label{exp:c_top}
    Tautological Chern classes of matroids also generalize Bergman classes: According to \cite[Thm.~7.6]{Berget_Eur_Spink_Tseng_2023}, there holds
    \begin{align*}
        c_{\crk(M)}(\umcQ_M)\cap\Delta_{\uSigma_E} = \uDelta_M.
    \end{align*}
    Therefore, the result from \cite[Thm.~5.7]{Huh_2016} shows that, if $M$ is loopless, then $c_{\crk(M)}(\umcQ_M)$ generates an extremal ray of $\Nef^{\crk(M)}(\uX_E)$.
    Since the structure of our proof for Theorem~\ref{thm:main_result} will be similar to the one of this result, let us summarize the steps made in \cite{Huh_2016}: By Proposition~\ref{prop:char_nef_rays}, it is sufficient to show that every weight from $\MW^{\crk(M)}(\uSigma_M)$ is constant. This was done as follows:
    \begin{enumerate}[label=(\Roman*)]
        \item For any $\Delta\in\MW^{\crk(M)}(\uSigma_M)$, consider the balancing condition at a flag $\mcF$ of length $\rk(M)-2$ consisting only of flats of $E$ and deduce that $\Delta$ is constant on all maximal flags of flats that contain $\mcF$ as a subflag.
        \item Conclude by the fact that Bergman fan $\uSigma_M$ is connected in codimension one, which is a consequence of the shellability of the order complex of the lattice of flats of $M$, proved in \cite{Bjoerner_1992}.
    \end{enumerate}
\end{example}

\subsection*{Proof of the main result} 
We are now ready to state and prove our main result:
\begin{mainthmm} \label{thm:main_result}
Let $M$ be a matroid on $E$, let $l$ be the number of loops and $l'$ be the number of coloops of $M$. Then, for any integers $k$ and $k'$ with $1\leq k<\crk(M)-l$ and $1\leq k'<\rk(M)-l'$, the following are equivalent:
\begin{enumerate}
    \item $M$ is connected up to loops and coloops, \label{eq:main_result_1}
    \item $c_k(\umcQ_M)$ generates an extremal ray of $\Nef^k(\uX_E)$, \label{eq:main_result_2}
    \item $c_{k'}(\umcS_M^\vee)$ generates an extremal ray of $\Nef^{k'}(\uX_E)$. \label{eq:main_result_3}
\end{enumerate}
\end{mainthmm}
Recall that the tautological sub and quotient Chern classes are related by the duality property from \eqref{eq:duality_property}. If we want to make use of this, we should understand how the nef cones of $\uX_E$ behave under the Cremona involution:
\begin{lemma} \label{lem:nef_crem_invariant}
    The nef cones of $\uX_E$ are Cremona-invariant and the Cremona involution preserves extremality of rays.
\begin{proof}
    By the projection formula, a pullback of a nef Chow class is again nef, which implies $\crem\bigl(\Nef^k(\uX_E)\bigr)\subset\Nef^k(\uX_E)$. Applying $\crem$ again proves equality. The second claim follows follows from the first one as $\crem$ is linear.
\end{proof}
\end{lemma}

\begin{corollary}
    The equivalence of \eqref{eq:main_result_1} and \eqref{eq:main_result_3} from Theorem~\ref{thm:main_result} is implied by the equivalence of \eqref{eq:main_result_1} and \eqref{eq:main_result_2}.
\begin{proof}
    By the duality property from \eqref{eq:duality_property}, we have $c_{k'}(\umcS_M^\vee)=\crem c_{k'}(\umcQ_{M^*})$. By Lemma~\ref{lem:nef_crem_invariant}, this generates an extremal ray if and only if $c_{k'}(\umcQ_{M^*})$ does. Assuming the equivalence of conditions \eqref{eq:main_result_1} and \eqref{eq:main_result_2}, this is the case if and only if $M^*$ is connected up to loops and coloops, which is equivalent to $M$ being connected up to loops and coloops.
\end{proof}
\end{corollary}

Next, we will show that it is sufficient to prove the equivalence of \eqref{eq:main_result_1} and \eqref{eq:main_result_2} for the case where $M$ is loopless and coloopless.
Whenever $\mcF=\bigl(F_1\subsetneq\cdots\subsetneq F_d\bigr)$ is a flag on $E$, we will for the rest of this section assume that every $F_1,\dots,F_d$ are non-empty and proper subsets of $E$ and we use the convention $F_0\coloneqq\emptyset$ and $F_{d+1}\coloneqq E$. If $F_{i+1}\sm F_i$ is of cardinality at least two for some $i\in[d]$, we say that this set is a $\textit{gap}$ of $\mcF$. We put
\begin{align*}
\gaps(\mcF)\coloneqq\bigl\{i\in[d] \,\big|\, |F_{i+1}\sm F_i|\geq 2\bigr\}.
\end{align*}
For a flag $\mcF$ on $E\cup\{e\}$, let $\mcF\sm e$ be the flag on $E$ that arises from $\mcF$ by removing $e$ from every set of $\mcF$, contracting duplicate sets and deleting the empty set and the set $E$ in case they arise in the process. Geometrically, one has $\pi(\sigma_{\mcF})=\sigma_{\mcF\sm e}$ where $\pi:\RR^{E\cup\{e\}}/\Span(\mathbf{e}_{E\cup\{e\}})\to\RR^E/\Span(\mathbf{e}_E)$ is the projection. Further, $e$ is contained in a gap of $\mcF$ if and only if $\dim(\sigma_\mcF)=\dim(\sigma_{\mcF\sm e})$.

\begin{lemma} \label{lem:pullback_subspace}
     Let $f:\uX_{E\cup\{e\}}\to\uX_{E}$ be the toric morphism induced from the projection $\pi$. Then, the subspace $f^*\CH^k(\uX_E)$ of $\CH^k(\uX_{E\cup\{e\}})$ is the orthogonal complement with respect to the degree pairing of the space spanned by $x_\mcF$, where $\mcF$ ranges over those flags of length $n+1-k$ such that $e$ is contained in a gap of $\mcF$.
\begin{proof}
    Let $x\in\CH^k(\uX_E)$ and $\mcF$ be a flag as in the claim. By the projection formula, we then have
    \begin{align*}
        \int_{\uX_{E\cup\{e\}}}x_\mcF\cdot f^*x=\int_{\uX_E}f_*x_\mcF\cdot x = 0,
    \end{align*}
    as $f_*x_\mcF=0$ since the dimension of $V(\sigma_\mcF)$ drops after taking the image under $f$.\\
    To see the other direction, suppose that $y\in\CH^k(\uX_{E\cup\{e\}})$ is orthogonal to all $x_\mcF$ as in the claim. In terms of Minkowski weights, this precisely means that no flag having a gap containing $e$ is in the support of $\Delta=y\cap\Delta_{\uSigma_E}$. Suppose $\mcG=\bigl(G_1\subsetneq\cdots\subsetneq G_{n-k}\bigr)$ is a flag such that it induces a cone from $\Sigma_{\Delta}(1)$ and $l$ is such that $e$ is contained in the gap $G_{l+1}\sm G_l$. Then, the balancing condition for $\Delta$ at $\mcG$ states that
    \begin{align*}
        &\Delta\bigl(G_1\subsetneq\cdots\subsetneq G_l\subsetneq G_l\cup\{e\}\subsetneq G_{l+1}\subsetneq\cdots\subsetneq G_{n-k}\bigr)\overline{\mathbf{e}_{G_l\cup\{e\}}}\\
        +\,&\Delta\bigl(G_1\subsetneq\cdots\subsetneq G_l\subsetneq G_{l+1}\sm\{e\}\subsetneq G_{l+1}\subsetneq\cdots\subsetneq G_{n-k}\bigr)\overline{\mathbf{e}_{G_{l+1}\sm\{e\}}}
        \in\Span(\sigma_\mcG).
    \end{align*}
    Clearly, this forces the two appearing values of $\Delta$ to be equal. This shows that, whenever we have $\mcF,\mcF'\in\Sigma_\Delta(0)$ with $\mcF\sm e=\mcF'\sm e$, then $\Delta(\mcF)=\Delta(\mcF')$ and hence $\Delta$ is the pullback of the Minkowski weight on $\uSigma_E$ that assigns to a flag $\mcH$ the value of $\Delta$ at any flag $\mcF$ with $\mcF\sm e=\mcH$.
\end{proof}
\end{lemma}

\begin{corollary} \label{cor:lifting_property_of_nef}
    Let $E'\subsetneq E$ and $f:\uX_E\to\uX_{E'}$ be the toric morphism induced from the projection $\RR^E/\Span(\mathbf{e}_E)\to\RR^{E'}/\Span(\mathbf{e}_{E'})$. For $x\in\CH^k(\uX_{E'})$, we have that $x$ is nef if and only if the pullback $f^*x$ is nef. The cone $f^*\Nef^k(\uX_{E'})$ is a face of $\Nef^k(\uX_E)$ and $x$ generates an extremal ray of $\Nef^k(\uX_{E'})$ if and only if $f^*x$ generates an extremal ray of $\Nef^k(\uX_E)$.
\begin{proof}
    If $x$ is nef, then $f^*x$ is nef as a pullback of a nef Chow class. Any effective Chow class on $\uX_{E'}$ can be written as a pushforward under $f$ of an effective Chow class on $\uX_E$. Using this and the projection formula, one sees that, if $f^*x$ is nef, then $x$ is nef. Hence,
    \begin{align*}
        f^*\Nef^k(\uX_{E'})
        = f^*\CH^k(\uX_{E'})\cap\Nef^k(\uX_E).
    \end{align*}
    By the previous lemma and the observation from~\eqref{eq:nef_for_toric}, it is clear that $f^*\Nef^k(\uX_{E'})$ is a face of $\Nef^k(\uX_E)$ and the last claim now follows directly from the injectivity of $f^*$.
\end{proof}
\end{corollary}

\begin{corollary} \label{cor:can_assume_(co)loopless}
    It is sufficient to prove the equivalence of \eqref{eq:main_result_1} and \eqref{eq:main_result_2} for matroids which are both loopless and coloopless.
\begin{proof}
    Let $L$ be a matroid consisting of a loop or a coloop $l$ and let $f:\uX_{E\cup\{l\}}\to\uX_E$ be as before. Since $\bigl(\Elem_i(B_\sigma(L))\bigr)_\sigma$ is a global polynomial that vanishes at the origin for $i>0$, we find $c_i(\umcQ_L)=0$ for $i>0$ and thus, by \cite[Prop.~5.13]{Berget_Eur_Spink_Tseng_2023}, we have
    \begin{align*}
        c_k(\umcQ_{M\oplus L})=\sum_{i=0}^kf^*c_i(\umcQ_M)c_{k-i}(\umcQ_L)=f^*c_k(\umcQ_M).
    \end{align*}
    By Corollary~\ref{cor:lifting_property_of_nef}, the class $c_k(\umcQ_{M\oplus L})$ generates an extremal ray if and only if $c_k(\umcQ_M)$ does. Since, tautologically, $M\oplus L$ is connected up to loops and coloops if and only if $M$ is, this inductively yields the claim.
\end{proof}
\end{corollary}

\begin{proposition}
     Condition \eqref{eq:main_result_2} implies condition \eqref{eq:main_result_1}: Let $M$ be a matroid on $E$ and suppose $1\leq k<\crk(M)-l$, where $l$ is the number of loops of $M$. If $c_k(\umcQ_M)$ generates an extremal ray of $\Nef^k(\uX_E)$, then $M$ is connected up to loops and coloops.
\begin{proof}
     By Corollary~\ref{cor:can_assume_(co)loopless}, we may assume $M$ to be loopless and coloopless. For contraposition, suppose $M=M'\oplus M''$, where $M'$ and $M''$ are matroids on $E'$ and $E''$ respectively. We consider the toric morphisms $f:\uX_E\to\uX_{E'}$ and $g:\uX_E\to\uX_{E''}$ induced from the respective projections of lattices. We infer from~\cite[Prop.~5.13]{Berget_Eur_Spink_Tseng_2023} that
    \begin{align*}
        c_k(\umcQ_M)=\sum_{i=0}^kf^*c_i(\umcQ_{M'})g^*c_{k-i}(\umcQ_{M''}).
    \end{align*}
    Since $1\leq k<\crk(M)=\crk(M')+\crk(M'')$, this sum has at least two non-zero summands. It follows from the fan displacement rule for Minkowski weights from \cite{Fulton_Sturmfels_1997} that these summands are nef as products of pullbacks of nef classes on a complete toric variety. Further, it is clear from the definition of tautological classes of matroids in terms of piecewise polynomial functions on the permutohedral fan that the summands are not multiples. Hence, the ray generated by $c_k(\umcQ_M)$ is not extremal.
\end{proof}
\end{proposition}
To prove the implication \eqref{eq:main_result_1} to \eqref{eq:main_result_2}, let us first introduce more notation concerning flags.
Whenever $\mcF\prec\mcG$, we define $\mcG\sm\mcF$ to be the flag that arises from $\mcG$ by removing all sets that also appear in $\mcF$ (here $\mcF$ can also be a set, interpreted as a flag of length one). If $S\subset E$, write $\mcF/S$ for the flag on $E\sm S$ constructed by deleting the elements of $S$ from every set of $\mcF$, contracting duplicate sets and removing empty sets as well as sets that became equal to $E\sm S$.
If $S$ is a non-empty proper subset of a gap $F_{i+1}\sm F_i$ of $\mcF$, let $\mcF\cup S$ be the flag that arises from $\mcF$ by inserting the set $F_i\cup S$ between $F_i$ and $F_{i+1}$.

Recall that we write $\mcF\in\Sigma(d)$ whenever $\sigma_\mcF\in\Sigma(d)$ for a subfan $\Sigma$ of $\uSigma_E$ and $\Delta(\mcF)$ for $\Delta(\sigma_\mcF)$ whenever $\Delta$ is a Minkowski weight on the permutohedral fan or a subfan.
For $1\leq k\leq\crk(M)$, let us denote $\uDelta_M^k\coloneqq c_k(\umcQ_M)\cap\Delta_{\uSigma_E}$ and $\uSigma_M^k\coloneqq\Sigma_{\uDelta_M^k}$. In view of Example~\ref{exp:c_top}, note that $\uDelta_M^{\crk(M)}$ is the Minkowski weight $\uDelta_M$ and, if $M$ is loopless, $\uSigma_M^{\crk(M)}$ is the Bergman fan $\uSigma_M$.
It is easy to see from Proposition~\ref{prop:Comb_Descr_MW_of_taut_classes} that the maximal cones of $\uSigma_M^k$ are given by
\begin{align} \label{eq:Sigma_M^k(n-k)}
    \hspace{2.2em}\mathclap{\uSigma_M^k(0)=\left\{\mcF=\bigl(F_1 \subsetneq \cdots \subsetneq F_{n-k}\bigr) \,\middle|\, M|F_{i+1}/F_i\text{ is uniform of rank one for all $i\in\gaps(\mcF)$}\right\}.}
\end{align}
\noindent If $M$ is loopless, the tautological class $c_k(\umcQ_M)$ is non-zero for $1\leq k\leq\crk(M)$ and hence $\uSigma_M^k(0)$ is non-empty. This is why, for the remainder of this section, we will implicitly always assume that $M$ is a loopless matroid.

Having this, we can now sketch the proof idea for the last implication of our main theorem. We will follow a strategy similar to the one from Example~\ref{exp:c_top}:
By Proposition~\ref{prop:char_nef_rays}, it will be sufficient to show that every Minkowski weight from $\MW^k(\uSigma_M^k)$ is constant. To this end, we will do the following:
\begin{enumerate}[label=(\Roman*)]
    \item
    Given $\Delta\in\MW^k(\uSigma_M^k)$ and $\mcF\in\uSigma_M^k(1)$, determine what we can infer about the values of $\Delta$ from the balancing condition at $\mcF$.
    \item
    Show by induction on $|E|$ that this information is sufficient to conclude that $\Delta$ is constant on $\uSigma_M^k(0)$.
\end{enumerate}

\subsection*{Part I}
If we want to consider balancing conditions in the fan $\uSigma_M^k$, we should first understand, for a given flag $\mcF\in\uSigma_M^k(1)$, the set of flags $\mcG\in\uSigma_M^k(0)$ with $\mcF\prec\mcG$ or, geometrically speaking, the star of $\sigma_\mcF$ in the fan $\uSigma_M^k$. Since any such flag $\mcG$ is of the form $\mcF\cup S$ for some set $S$ contained in a gap of $\mcF$, this is the same as asking for a description of
\begin{align*}
S_M^k(\mcF)\coloneqq\{S\mid\emptyset\neq S\subsetneq F_{i+1}\sm F_i \text{ for some $i\in\gaps(\mcF)$ such that }\mcF\cup S\in\uSigma_M^k(0)\}.
\end{align*}

\begin{lemma}\label{lem:allowed_moves}
    Let $M$ be a matroid on $E$ and let $1\leq k\leq\crk(M)$. For $\mcF\in\uSigma_M^k(1)$, there holds:
    \begin{itemize}
        \item If $\mcF\in\uSigma_M^{k+1}(0)$, then $S_M^k(\mcF)=\{F_{i+1}\sm(F_i\cup\{e\})\mid i\in\gaps(\mcF), e\in F_{i+1}\sm F_i\}$,
        \item If $\mcF\notin\uSigma_M^{k+1}(0)$, then $S_M^k(\mcF)$ is a proper partition of $F_{i+1}\sm F_i$, where $i\in\gaps(\mcF)$ is the unique index for which $M|F_{i+1}/F_i$ is not uniform of rank one.
    \end{itemize}
\begin{proof}
    The case where $\mcF\in\uSigma_M^{k+1}(0)$ is immediate from \eqref{eq:Sigma_M^k(n-k)}. Thus, let $\mcF\notin\uSigma_M^{k+1}(0)$ and let $i$ be the unique index such that $N = M|F_{i+1}/F_i$ is not uniform of rank one. As $\mcF\in\uSigma_M^k(1)$, the minor $N$ is of rank at most two. If $N$ is of rank zero, it has to be the sum of two loops and thus $S_M^k(\mcF)$ is the partition into the two singletons. If $N$ is of rank one, it has exactly one loop $l$ and $S_M^k(\mcF)$ is the partition into the subsets $\{l\}$ and $F_{i+1}\sm(F_i\cup\{l\})$. Lastly, if $N$ is of rank two, note that $S\in S_M^k(\mcF)$ if and only if $S$ is a flat of rank one of $N$. Clearly, the rank one flats of $N$ cover $F_{i+1}\sm F_i$. The intersection of two distinct rank one flats of $N$ is of rank zero and therefore empty, as $N$ must be loopless because of~$\mcF\in\uSigma_M^k(1)$.
\end{proof}
\end{lemma}

As pointed out in Example~\ref{exp:c_top}, in the case where $k=\crk(M)$, the balancing condition at $\mcF$ directly implied that every Minkowski weight on $\uSigma_M$ of codimension $k$ is constant on cones having $\sigma_\mcF$ as a face. However, if $k<\crk(M)$, this is in general no longer the case:

\begin{example} \label{exp:bal_cond_counterex}
    Consider $E=\{0,1,2,3\}$ and the connected matroid $M$ whose bases are all cardinality two subsets of $E$ apart from $\{0,1\}$. The flag $\mcF=\bigl(\{0,1\}\bigr)$ lies in $\uSigma_M^1(1)$ and the balancing condition for a weight $\Delta\in\MW^1(\uSigma_M^1)$ at $\mcF$ states
    \begin{align*}
        &\Delta(\{0\}\subsetneq\{0,1\})\overline{\mathbf{e}_0} + \Delta(\{1\}\subsetneq\{0,1\})\overline{\mathbf{e}_1}\\
        +\,&\Delta(\{0,1\}\subsetneq\{0,1,2\})\overline{\mathbf{e}_{\{0,1,2\}}} + \Delta(\{0,1\}\subsetneq\{0,1,3\})\overline{\mathbf{e}_{\{0,1,3\}}} \in \Span(\overline{\mathbf{e}_{\{0,1\}}}).
    \end{align*}
    As long as the first two and last two values of $\Delta$ are equal, this equation is satisfied and we cannot immediately conclude that all four of them must be equal. However, this can still be done, for instance, by considering the balancing conditions at the rays corresponding to the sets $\{1\}$, $\{1,2,3\}$, $\{3\}$ and $\{0,1,3\}$. If we identify $\RR^E/\Span(\mathbf{e}_E)$ with $\RR^3$ by mapping $\overline{\mathbf{e}_i}$ to $\mathbf{e}_i$ for $i=1,2,3$, this situation can be visualized as follows:
    \end{example}
    \begin{figure}[h]
        \centering
        \includegraphics[width=0.4\linewidth]{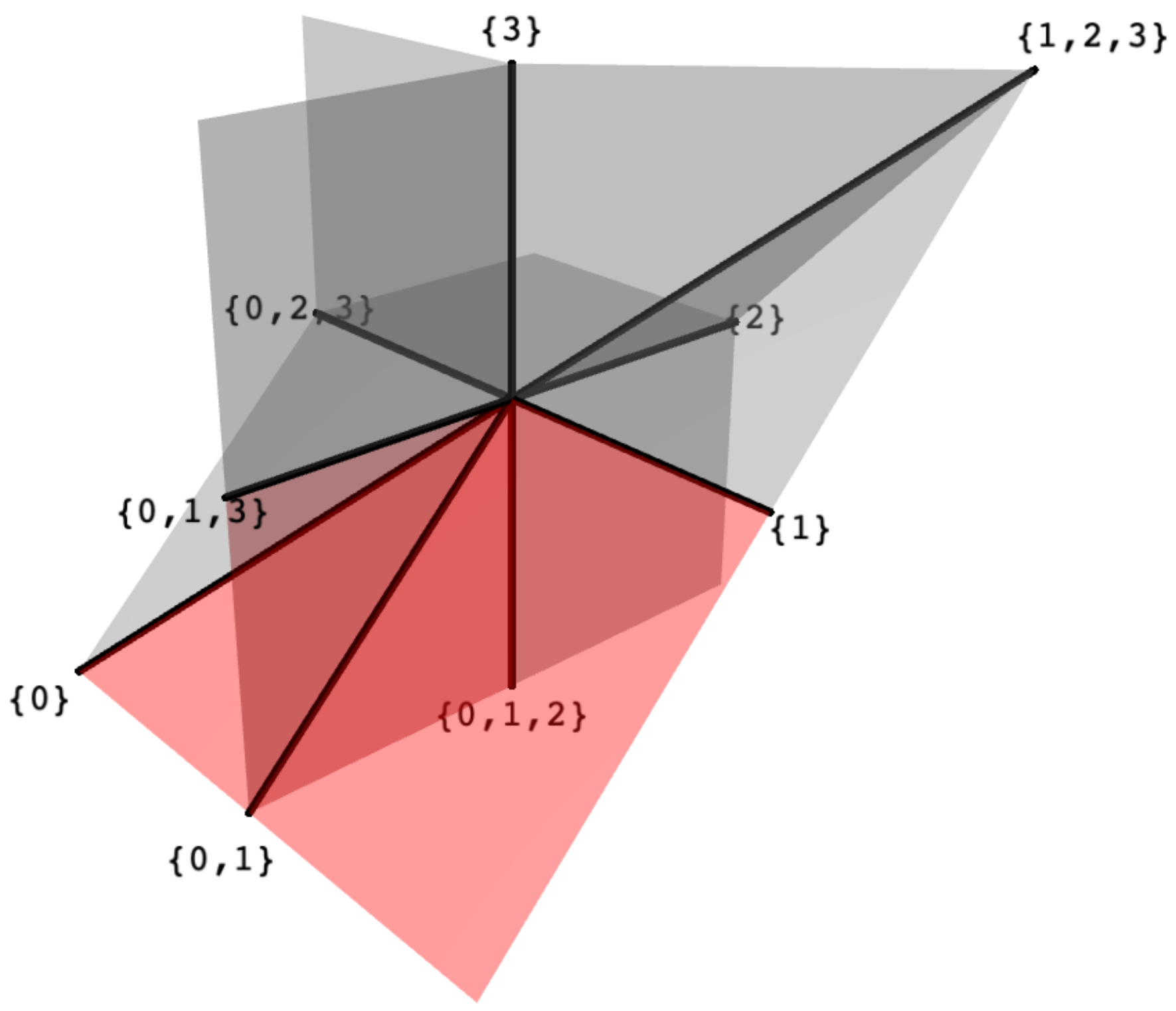}
        \caption{$\uSigma_M^1$ with $\torstar_{\overline{\mathbf{e}_{\{0,1\}}}}(\uSigma_M^1)$ highlighted in \textcolor{red}{red}.}
    \end{figure}

This example suggests that the balancing condition at $\mcF\in\uSigma_M^k(1)$ can only relate the values of $\mcG=\mcF\cup S$ and $\mcG'=\mcF\cup S'$ if $S$ and $S'$ are subsets of a common gap of $\mcF$. This observation leads us the following definition and proposition:

\begin{definition} \label{def:equivalence_relation}
    Let $M$ be matroid on $E$ and let $1\leq k\leq\crk(M)$. 
    We define $\sim_M^k$ to be the smallest equivalence relation on $\uSigma_M^k(0)$ such that $\mcF\cup S\sim_M^k\mcF\cup S'$ whenever $\mcF\in\uSigma_M^k(1)$ and $S,S'\in S_M^k(\mcF)$ are subsets of a common gap of $\mcF$. If no confusion is likely, we will suppress the indices and simply write $\sim$ instead of $\sim_M^k$.
\end{definition}

\begin{proposition}\label{prop:bal_cond_implies_loc_const}
    Let $M$ be a matroid on $E$ and $1\leq k\leq\crk(M)$. Then, any Minkowski weight from $\MW^k(\Sigma_M^k)$ is constant on $\sim$-equivalence classes.
\begin{proof}
    The balancing condition for $\Delta\in\MW^k(\Sigma_M^k)$ at $\mcF=\bigl(F_1\subsetneq\cdots\subsetneq F_{n-k-1}\bigr)$ states that
    \begin{align} \label{eq:bal_con_at_F}
        \sum_{S\in S_M^k(\mcF)}\Delta(\mcF\cup S)(\overline{\mathbf{e}_{(\mcF\cup S)\sm\mcF}}) \in\Span(\sigma_\mcF).
    \end{align}
    As $(\overline{\mathbf{e}_{F_{i+1}\sm F_i}})_{i\in[n-k-2]}$ is a basis of $\Span(\sigma_\mcF)$, there are unique $c_i\in\QQ$ with
    \begin{align*}
        \sum_{S\in S_M^k(\mcF)}\Delta(\mcF\cup S)\overline{\mathbf{e}_S} = \sum_{i=0}^{n-k-2} c_i\overline{\mathbf{e}_{F_{i+1}\sm F_i}}.
    \end{align*}
    If $\mcF\notin\uSigma_M^{k+1}(0)$, then $S_M^k(\mcF)$ is a partition of $F_{l+1}\sm F_l$ for some $l\in\gaps(\mcF)$ by Lemma~\ref{lem:allowed_moves}. If further $l\neq n-k-1$, all values of $\Delta$ appearing above are equal to $c_l$ and $c_i$ is zero for $i\neq l$. In the case $l=n-k-1$, we have $c\coloneqq c_1=\cdots=c_{n-k-2}$ and the appearing values of $\Delta$ are equal to $-c$.
    If $\mcF\in\uSigma_M^{k+1}(0)$, the condition from \eqref{eq:bal_con_at_F} can be expressed as
    \begin{align*}
        \sum_{i\in\gaps(\mcF)}\sum_{j\in F_{i+1}\sm F_i}\Delta\bigl(\mcF\cup(F_{i+1}\sm(F_i\cup\{j\}))\bigr)\overline{\mathbf{e}_j}\in\Span(\sigma_\mcF).
    \end{align*}
    Now one can argue similarly as in the previous case to see that, for a fixed $i\in\gaps(\mcF)$, the values of $\Delta$ on $\mcF\cup(F_{i+1}\sm(F_i\cup\{j\}))$ are equal for $j\in F_{i+1}\sm F_i$. By our definition of the relation~$\sim$, this completes the proof.
\end{proof}
\end{proposition}

\subsection*{Part II} It is evident from Example~\ref{exp:bal_cond_counterex} that, without any additional arguments, proving connectedness in codimension one of $\uSigma_M^k$ does not suffice to establish that all weights from the space $\MW^k(\uSigma_M^k)$ are constant. In view of Proposition~\ref{prop:bal_cond_implies_loc_const}, we should instead prove the stronger statement that all flags from $\uSigma_M^k(0)$ are $\sim$-equivalent. This will require some preparation. Recall that, for a matroid $M$ on $E$, two elements $e,f\in E$ are called \textit{parallel} in $M$, if $\{e,f\}$ is a circuit of $M$. A \textit{parallel class} of $M$ is a maximal subset of $E$ such that any two distinct elements are parallel. A parallel class is called \textit{trivial}, if it contains only one element. Whenever $M$ is connected, a subset $S\subset E$ is called \textit{contractible} in $M$, if $M/S$ is again connected. We say that $e$ is contractible in $M$, if the set $\{e\}$ is contractible in $M$.

\begin{lemma} \label{lem:downarrow_flag}
    Let $M$ be a matroid on $E$, let $1\leq k\leq\crk(M)$ and fix a labeling $E\to[n]$. Let further $\mcG=\bigl(G_1\subsetneq\cdots\subsetneq G_{n-k}\bigr)\in\uSigma_M^k(0)$ and $e\in E$. Then there is a flag \begin{align*}\da_e\mcG=\bigl(\da_e G_1\subsetneq\cdots\subsetneq\da_eG_{n-k}\bigr)\in\uSigma_M^k(0)\end{align*} dependent on the labeling of $E$ such that $\mcG\sim\da_e\mcG$, the set $\da_eG_1$ contains $e$ and, if $l$ is the maximal index with $G_l$ not containing $e$, then $\da_eG_i=G_i$ for all $i>l$. Further, the set $\da_eG_1$ is then contained in the parallel class of $e$.
\begin{proof}
    If $l=0$, put $\da_e\mcG\coloneqq\mcG$. Otherwise, we apply Lemma~\ref{lem:allowed_moves} to $\mcF\coloneqq\mcG\sm G_l$ which shows that there is $S\in S_M^k(\mcF)$ with $S\subset G_{l+1}\sm G_{l-1}$ and $e\in S$. If $\mcF\notin\uSigma_M^{k+1}(0)$, this choice is unique. If $\mcF\in\uSigma_M^{k+1}(0)$, we additionally require the label of the only element of $G_{l+1}\sm(G_{l-1}\cup S)$ to be maximal, which then determines us a unique set. In either case, we find $\mcG\sim\mcF\cup S$ and $e$ lies in the $l$-th set of $\mcF\cup S$. It is clear from the construction that the flag resulting from applying this procedure inductively meets the first three conditions made in the claim. Since $\da_e\mcG\in\uSigma_M^k(0)$, we have $M|\da_eG_1=U_1(\da_eG_1)$ and hence all elements from $\da_eG_1$ are parallel to $e$.
\end{proof}
\end{lemma}

Although the construction of $\da_e\mcG$ was dependent on a labeling of $E$, its only purpose was to be able to make distinguished choices along the proof to end up with a unique flag with the required properties. These properties are the only fact we will be using in following, which is why we will omit mentioning any labelings.

\begin{lemma} \label{lem:existence_breaking_flag}
    Let $M$ be a matroid on $E$ and $1\leq k<\crk(M)$. Then, for any $e\in E$ with non-trivial parallel class $P$, there is a flag $\mcG=\bigl(G_1\subsetneq\cdots\subsetneq G_{n-k}\bigr)\in\uSigma_M^k(0)$ with $G_1=P\sm\{e\}$ and $G_2=P$.
    \begin{proof}
        Since $M$ is loopless by our standing assumption and $\crk(M/P)\geq k-|P|+2$, there is a flag $\bigl(F_1\subsetneq\cdots\subsetneq F_{n-k-2}\bigr)$ in $\uSigma_{M/P}^{k-|P|+2}(0)$. Having this, it is clear that the flag
        \[\mcG=\bigl(P\sm\{e\}\subsetneq P\subsetneq P\cup F_1\subsetneq\cdots\subsetneq P\cup F_{n-k-2}\bigr)\]
        has the desired properties.
    \end{proof}
\end{lemma}

\begin{lemma} \label{lem:move_parallel_elements}
    Let $M$ be a matroid on $E$ and let $1\leq k\leq\crk(M)$. Suppose we have $e\in E$ with non-trivial parallel class $P$ and $\mcG=\bigl(G_1\subsetneq\cdots\subsetneq G_{n-k}\bigr)\in\uSigma_M^k(0)$ such that $G_1\subsetneq P$. For every $i=1,\dots,n-k$, there is a flag $\mcH=\bigl(H_1\subsetneq\cdots\subsetneq H_{n-k}\bigr)\in\uSigma_M^k(0)$ with $\mcG\sim\mcH$ and $H_{i+1}\sm H_i=\{e\}$.
    \begin{proof}
        If $e\notin G_1$, it is follows from $\mcG\in\uSigma_M^k(0)$ that there is an index $j\in\{1,\dots,n-k\}$ such that $G_{j+1}\sm G_j=\{e\}$. If $j\geq2$, we can apply Lemma~\ref{lem:allowed_moves} to $\mcG\sm G_j$ and find $\mcH\in\uSigma_M^k(0)$ with $\mcH\sim\mcG$ and $H_j\sm H_{j-1}=\{e\}$. Similarly, if $j\leq n-k-1$, we apply Lemma~\ref{lem:allowed_moves} to $\mcG\sm G_{j+1}$ and obtain $\mcH\in\uSigma_M^k(0)$ with $\mcH\sim\mcG$ and $H_{j+2}\sm H_{j+1}=\{e\}$. Using this iteratively completes the proof for this case.\\
        Now assume $e\in G_1$. Since $G_1\neq P$, pick $f\in P\sm G_1$. By the previous case, we can assume that $G_2\sm G_1=\{f\}$. Applying Lemma~\ref{lem:allowed_moves} to $\mcG\sm G_1$ shows that $\mcG$ is $\sim$-equivalent to the flag in $\uSigma_M^k(0)$ that arises from $\mcG$ by swapping $e$ and $f$. Thus, we conclude by the previous case.
    \end{proof}
\end{lemma}

\begin{lemma}\label{lem:SuppOfContr}
    Let $M$ be a matroid on $E$ and $1\leq k\leq\crk(M)$. Let further $\mcG=\bigl(G_1\subsetneq\cdots\subsetneq G_{n-k}\bigr)$ be a flag on $E$ such that $G_1$ is a parallel class $P$ of $M$. Then, there holds
    \begin{align*}
        \mcG/P\in\uSigma_{M/P}^{k-|P|+1}(0) &\iff \mcG\in\uSigma_M^k(0).
    \end{align*}
    For two such flags $\mcG,\mcG'\in\uSigma_M^k(0)$, we further have that $\mcG/P\sim_{M/P}^{k-|P|+1}\mcG'/P$ implies $\mcG\sim_M^k\mcG'$.
\begin{proof}
    Since $P$ is a parallel class, the minor $M|P$ is uniform of rank one. Thus, the flag $\mcG$ lies in $\uSigma_M^k(0)$ if and only if the minor $M|G_{i+1}/G_i$ is uniform of rank one for all $i\in\gaps(\mcG)$ greater than zero.
    As this minor can rewritten as $M/P|(G_{i+1}\sm P)/(G_i\sm P)$, the previous condition is equivalent to $\mcG/P\in\uSigma_{M/P}^{k-|P|+1}(0)$. To see the second claim, note that, for any $\mcF\in\uSigma_{M/P}^{k-|P|+1}(1)$, the flag $P\mcF$, constructed by adding $P$ to every set from $\bigl(\emptyset\subsetneq\mcF\bigr)$, is in $\uSigma_M^k(1)$ and has the property $S_{M/P}^{k-|P|+1}(\mcF)\subset S_M^k(P\mcF)$. In view of Definition~\ref{def:equivalence_relation}, this yields the second claim.
\end{proof}
\end{lemma}

\begin{lemma}\label{lem:SuppOfDel}
    Let $M$ be a matroid on $E$ and $1\leq k\leq\crk(M)$. Let further $\mcG=\bigl(G_1\subsetneq\cdots\subsetneq G_{n-k}\bigr)$ be a flag on $E$ such that $G_{n-k}=E\sm\{e\}$ for some $e\in E$. Then there holds
    \begin{align*}
        \mcG\sm e\in\uSigma_{M\sm e}^k(0) &\iff \mcG\in\uSigma_M^k(0).
    \end{align*}
    For two such flags $\mcG,\mcG'\in\uSigma_M^k(0)$, we further have that $\mcG\sm e\sim_{M\sm e}^k\mcG'\sm e$ implies $\mcG\sim_M^k\mcG'$.
    \begin{proof}
        Just as in the previous Lemma, the first claim follows from the fact that there holds $M\sm e|G_{i+1}/G_i=M|G_{i+1}/G_i$ for all $i\in[n-k-1]$. For $\mcF\in\uSigma_{M\sm e}^k(1)$, the flag $\bigl(\mcF\subsetneq E\sm\{e\}\bigr)$ is in $\uSigma_M^k(1)$ and one has $S_{M\sm e}^k(\mcF)\subset S_M^k(\mcF\subsetneq E\sm\{e\})$. This implies the second claim.
    \end{proof}
\end{lemma}

We are now prepared to finish part~II of the proof:

\begin{proposition} \label{prop:connectedness}
    Let $M$ be a connected matroid on $E$ and $1\leq k<\crk(M)$. Then all flags from $\uSigma_M^k(0)$ are $\sim$-equivalent.
    \begin{proof}
        For $|E|\leq 2$, there is nothing to prove. Thus, let $|E|\geq3$, i.e., $n=|E|-1\geq 2$. If $M=U_1(E)$, note that $\uSigma_{U_1(E)}^k(0)$ consists of all flags on $E$ whose sets are of cardinality $k+1,\dots,n$. They can directly be seen to be $\sim$-equivalent by Lemma~\ref{lem:allowed_moves}. Therefore, let us assume that $M$ is not uniform of rank one.
        We will now proceed by induction on the cardinality of the ground set of the matroid $M$.
        For $n=2$, the only connected matroid on $E$ of corank greater than one is $U_1(E)$, which is why we may further assume $n\geq 3$.\\
        \underline{Claim $1$}: Suppose $\mcG,\mcG'\in\uSigma_M^k(0)$ are such that there is a contractible parallel class $P$ of $M$ with $G_1,G_1'\subset P$. If $G_1=G_1'=P$ or $G_1,G_1'\subsetneq P$, then $\mcG\sim_M^k\mcG'$.\\
        \textit{Proof.}\hspace{0.5em}If $G_1=G_1'=P$, we have $\mcG/P\sim_{M/P}^{k-|P|+1}\mcG'/P$ by induction hypothesis, since $M/P$ is again connected and $\crk(M/P)=\crk(M)-(|P|-1)>k-|P|+1$. Hence, $\mcG\sim_M^k\mcG'$ by Lemma~\ref{lem:SuppOfContr}.
        If $G_1,G_1'\subsetneq P$, the parallel class $P$ cannot be trivial. Let $e\in P$ and, by Lemma~\ref{lem:move_parallel_elements}, choose $\mcH,\mcH'\in\uSigma_M^k(0)$ such that $\mcG\sim_M^k\mcH$, $\mcG'\sim_M^k\mcH'$ and $E\sm H_{n-k}=E\sm H_{n-k}'=\{e\}$. If $\crk(M)>k+1$, we infer from the induction hypothesis that $\mcH\sm e\sim_{M\sm e}^k\mcH'\sm e$, since $M\sm e$ is connected ($P$ was non-trivial) and $\crk(M\sm e)=\crk(M)-1>k$. In this case, we obtain $\mcG\sim_M^k\mcG'$ from Lemma~\ref{lem:SuppOfDel}. If however $\crk(M)=k+1$, only one minor from $M|H_{i+1}/H_i$ for $i\in[n-k]$ is a loop. By construction, this loop appears at $i=n-k$. As we assumed that $M$ is not uniform of rank one, there is $f\in E\sm P$ and we know that the first set from $\da_f\mcH$ is the whole parallel class of $f$. The same holds for $\da_f\mcH'$ and thus they are $\sim_M^k$-equivalent by the previous case.\qed\\
        To come back to our original statement, let $\mcG,\mcG'\in\uSigma_M^k(0)$. If there is $e\in E$ such that $\da_eG_1$ and $\da_eG'_1$ are both the entire parallel class of $e$ or are both a proper subset of the parallel class of $e$, then we have $\da_e\mcG'\sim_M^k\da_e\mcG'$ by the above claim and thus $\mcG\sim_M^k\mcG'$. Hence, let us assume that, for every $e\in E$ with contractible parallel class $P$ in $M$, we have $\da_eG_1=P$ if and only if $\da_eG'_1\neq P$. In particular, every element lies in a non-trivial parallel class.\\
        \underline{Claim $2$}: The matroid $M$ admits at least three distinct contractible parallel classes.\\
        \textit{Proof.}\hspace{0.5em}Suppose $M$ has at most two contractible parallel classes. By \cite[Lem.~3.6]{Wu_2005}, the set of contractible elements in the simplification of $M$ is spanning and hence $\rk(M)\leq2$. Since we assumed $M$ to be connected and not uniform of rank one, $M$ is a parallelization of a uniform matroid of rank two with at least three elements. As every element of a uniform matroid of rank two is contractible, $M$ cannot have had only two contractible parallel classes.\qed\\
        By the claim and our assumption, we may choose two parallel classes $P_1$ and $P_2$ of $M$ and $e_i\in P_i$ for $i=1,2$ such that, without loss of generality, there holds $\da_{e_1}G_1=P_1$, $\da_{e_2}G_1=P_2$ and $\da_{e_1}G'_1\neq P_1$.
        By Lemma~\ref{lem:existence_breaking_flag}, there is $\mcH\in\uSigma_{M/P_2}^{k-|P_2|+1}(0)$ with $(\downarrow_{e_2}\mcG)/P_2\sim_{M/P_2}^{k-|P_2|+1}\mcH$ as well as $H_1=P_1\sm\{e_1\}$ and $H_2=P_1$. Thus, we obtain $\downarrow_{e_2}\mcG\sim_M^k P_2\mcH$ from Lemma~\ref{lem:SuppOfContr}, where $P_2\mcH$ is the flag resulting from adding $P_2$ to every set of $\bigl(\emptyset\subsetneq\mcH\bigr)$. For any $f\in P_1\sm\{e_1\}$, we know that $e_1$ is not in the first set of $\da_f P_2\mcH$ by Lemma~\ref{lem:downarrow_flag} and thus we conclude $\da_{f}P_2\mcH\sim_M^k\da_{e_1}\mcG'$ by claim~$1$. Altogether, this implies $\mcG\sim_M^k\mcG'$.
    \end{proof}
\end{proposition}

This completes the proof for the remaining implication of \eqref{eq:main_result_1} to \eqref{eq:main_result_2} from Theorem~\ref{thm:main_result}: Corollary~\ref{cor:can_assume_(co)loopless} reduced to the case where $M$ is loopless and coloopless and, in this situation, the combination of Proposition~\ref{prop:bal_cond_implies_loc_const} and Proposition~\ref{prop:connectedness} shows that every Minkowski weight from $\MW^k(\uSigma_M^k)$ is constant for any $1\leq k<\crk(M)$. Applying Proposition~\ref{prop:char_nef_rays} proves the extremality of the ray generated by $c_k(\umcQ_M)$ in $\Nef^k(\uX_E)$.

As mentioned earlier, connectedness in codimension one of $\uSigma_M^k$ is implied by $\sim$-equivalence of all cones from $\uSigma_M^k(0)$. Thus, we have the following corollary:

\begin{corollary}
    For a connected matroid $M$ and $1\leq k\leq\crk(M)$, the fan $\uSigma_M^k$ is connected in codimension one.
\end{corollary}

In view of Example~\ref{exp:c_1}, our Theorem~\ref{thm:main_result} also recovers the result from \cite[Thm.~2.1.5]{Nguyen_1978}:

\begin{corollary}
    A matroid $M$ is connected up to loops and coloops if and only if its rank function $\rk_M$ generates an extremal ray of the cone of submodular functions modulo modular functions.
\end{corollary}

\section{Extension to the stellahedral variety} \label{sec:main_result_stella}

The \textit{stellahedron} is the lattice polytope in $\ZZ^E$ given by
\begin{align*}
    \Pi_E\coloneqq\{x\in\RR_{\geq0}^E\mid\exists y\in\uPi_E':y+\mathbf{e}_E-x\in\RR_{\geq0}^E\}.
\end{align*}
It is clear from this description that $\Pi_E$ has the translated permutohedron $\uPi_E'+\mathbf{e}_E$ as a facet. One can also give another description of $\Pi_E$ in terms of simplices, similar to the one for the permutohedron from \eqref{eq:Perm_as_hypersimplices}:
\begin{align} \label{eq:Stell_as_hypersimplices}
\Pi_E=\sum_{i=1}^{n+1}\Delta_{E,i}\text{, where }\Delta_{E,i}=\{x\in[0,1]^E\mid x_0+\cdots+x_n\leq i\}.
\end{align}
By definition, the \textit{stellahedral fan} $\Sigma_E$ is the normal fan of $\Pi_E$. As the permutohedron is a facet of the stellahedron, one could expect that a cone in $\Sigma_E$ can be described by a flag and some additional data: For a subset $I\subset E$ and a possibly empty flag $\mcF$ consisting of proper subsets of the set $E$, a pair $(I,\mcF)$ is called \textit{compatible}, if $I$ is contained in every set from $\mcF$. In this case, one writes $I\leq\mcF$ instead of $(I,\mcF)$. This definition allows us give the description of the cones of the stellahedral fan $\Sigma_E$ from \cite{Braden_Huh_Matherne_Proudfoot_Wang_2022}:

\begin{proposition}
    The set $\Sigma_E(k)$ of $k$-codimensional cones of $\Sigma_E$ is in bijection with the set of compatible pairs $I\leq\mcF$, where $\mcF=\bigl(F_0\subsetneq\cdots\subsetneq F_l\bigr)$ and $|I|+l+1=n+1-k$. The correspondence is given by
    \begin{align*}
        I\leq\mcF\mapsto \sigma_{I\leq\mcF}\coloneqq\cone(\mathbf{e}_i\mid i\in I)+\cone(-\mathbf{e}_{E\sm F_i}\mid i\in[l]).
    \end{align*}
\end{proposition}
Since $\Sigma_E$ is smooth, it induces a complete smooth toric variety $X_E\coloneqq X_{\Sigma_E}$, called the \textit{stellahedral variety}. As $\uPi_E$ is a facet of $\Pi_E$, it comes with a toric embedding $\iota_E:\uX_E\to X_E$.
Just as for the permutohedral variety, there are explicit combinatorial characterizations for the Chow ring $\CH^\bullet(X_E)$ that were exploited by the authors of \cite{Braden_Huh_Matherne_Proudfoot_Wang_2022} and \cite{Eur_Huh_Larson_2023} to define the \textit{augmented Bergman class} $\Delta_M\cap[X_E]$ and \textit{augmented tautological sub (resp. quotient) Chern classes} $c_k(\mcS_M)$ (resp. $c_k(\mcQ_M)$) of a matroid $M$ on $E$.
The key ingredient for the proof of the extension of our main result to the stellahedral variety will again be the combinatorial formula for the values of the Minkowski weights associated to augmented tautological Chern classes from \cite[Prop.~5.9]{Eur_Huh_Larson_2023}. Using the notation $\Delta_M^k\coloneqq c_k(\mcQ_M)\cap\Delta_{\Sigma_E}$, this can be stated as follows:

\begin{proposition} \label{prop:Comb_Descr_MW_of_taut_classes_stella}
    Let $M$ be a matroid on $E$ and $k\in[n]$. Let $I\leq\mcF$ be a compatible pair with $\mcF=\bigl(F_0\subsetneq\cdots\subsetneq F_l\bigr)$ and $l=n-k-|I|$. We put $F_{l+1}=E$ and interpret $F_0$ as $E$ if $\mcF$ is empty. Then, we have
    \begin{align*}
        \Delta_M^k(\sigma_{I\leq\mcF})
        &=\begin{cases}
            1 & \begin{aligned}
            &\text{$F_0\subset\cl_M(I)$ and for $i\in[l]$ exactly $l+\rk_M(I)-\rk(M)+1$} \\ &\text{minors $M|F_{i+1}/F_i$ are loops and the rest are uniform of rank one}
            \end{aligned}\\
            0 & \text{otherwise}
        \end{cases}.
    \end{align*}
    In particular, $c_k(\mcQ_M)$ is nef.
\end{proposition}

Having this, we are ready to prove our second main result, the extension of Theorem~\ref{thm:main_result} to the stellahedral variety:

\begin{mainthmm} \label{thm:main_result_stella}
If $M$ is loopless, the three conditions from Theorem~\ref{thm:main_result} are equivalent to
\begin{enumerate}[start = 4]
    \item $c_k(\mcQ_M)$ generates an extremal ray of $\Nef^k(X_E)$. \label{eq:main_result_4}
\end{enumerate}
\end{mainthmm}
Unless stated otherwise, we will assume $M$ to be loopless from now on. We will write $\Sigma_M^k$ for the fan $\Sigma_{\Delta_M^k}$ whose maximal cones are precisely those that are in the support of $\Delta_M^k$. By using notations from the previous chapter and Proposition~\ref{prop:Comb_Descr_MW_of_taut_classes_stella}, this precisely means
\begin{align} \label{eq:Sigma_M^k(n-k)_stella}
    \Sigma_M^k(0)=\left\{I\leq\mcF \,\middle|\,
    F_0\subset\cl_M(I) \text{ and } \mcF\text{ is empty or } \mcF/F_0\in\uSigma_{M/F_0}^{k+|I|-|F_0|}(0)\right\}.
\end{align}
For a compatible pair $(I\leq\mcF)\in\Sigma_M^k(1)$, let us consider the sets
\begin{align*}
    S_M^k(I\leq\mcF)&\coloneqq\{S\subsetneq F_0 \mid I\leq(S\subsetneq\mcF)\in\Sigma_M^k(0)\},\\
    E_M^k(I\leq\mcF)&\coloneqq\{e\in F_0\sm I \mid I\cup\{e\}\leq\mcF\in\Sigma_M^k(0)\}.
\end{align*}
Note that, if $\mcF$ is non-empty and $\mcF/F_0\notin\uSigma_{M/F_0}^{k+|I|-|F_0|}(0)$, there is $i\in\gaps(\mcF)$ such that $M|F_{i+1}/F_i$ is not uniform of rank one and hence both of the above sets must be empty in this case.

\begin{lemma}\label{lem:S_M^k_and_E_M^k}
    Let $M$ be a matroid on $E$ and $1\leq k\leq\crk(M)$. Suppose $I\leq\mcF\in\Sigma_M^k(1)$ is such that $\mcF$ is empty or $\mcF/F_0\in\uSigma_{M/F_0}^{k+|I|-|F_0|}(0)$. Then, there holds:
    \begin{align*}
        S_M^k(I\leq\mcF)=\begin{cases}
            \{F_0\sm\{e\}\mid e\in F_0\sm I\} & \text{if }F_0\subset\cl_M(I)\\
            \{\cl_{M|F_0}(I)\} & \text{otherwise}
        \end{cases}
    \end{align*}
    and
    \begin{align*}
        E_M^k(I\leq\mcF)=\begin{cases}
            F_0\sm I & \text{if }F_0\subset\cl_M(I)\\
            F_0\sm\cl_M(I) & \text{otherwise}
        \end{cases}.
    \end{align*}
\begin{proof}
    For $S\in S_M^k(I\leq\mcF)$, the minor $N\coloneqq M|F_0/S$ is a loop or uniform of rank one. If $F_0\subset\cl_M(I)$, then $N$ has to be a loop, so $S=F_0\sm\{e\}$ for some $e\in F_0\sm I$. By the assumptions made in the claim, every set of this form does indeed lie in $S_M^k(I\leq\mcF)$. If $F_0\not\subset\cl_M(I)$, the minor $N$ is uniform of rank one. This implies $S=\cl_{M|F_0}(I)$, as, if either inclusion fails, $N$ would have loops. It is obvious that $\cl_{M|F_0}(I)\in S_M^k(I\leq\mcF)$, which completes the proof of the first part.\\
    For $e\in F_0\sm I$, we have $e\in E_M^k(I\leq\mcF)$ if and only if $F_0\subset\cl_M(I\cup\{e\})$. If $F_0\subset\cl_M(I)$, this yields $E_M^k(I\leq\mcF)=F_0\sm I$. If $F_0\not\subset\cl_M(I)$, we we see from $I\leq\mcF\in\Sigma_M^k(1)$ that $\rk_M(I)=\rk_M(F_0)-1$ and thus  $E_M^k(I\leq\mcF)=F_0\sm\cl_M(I)$.
\end{proof}
\end{lemma}

\begin{definition} \label{def:approx_relation}
    Let $M$ be matroid on $E$ and let $1\leq k\leq\crk(M)$.
    Let $\approx_M^k$ be the smallest equivalence relation on $\Sigma_M^k(0)$ such that, for all compatible pairs $I\leq\mcF\in\Sigma_M^k(1)$ with $\mcF$ empty or $\mcF/F_0\in\uSigma_{M/F_0}^{k+|I|-|F_0|}(0)$ and for all $e\in E_M^k(I\leq\mcF)$ and $S\in S_M^k(I\leq\mcF)$, we have 
    \begin{align*}
    I\cup\{e\}\leq\mcF\approx_M^k I\leq(S\subsetneq\mcF).
    \end{align*}
    If no confusion is likely, we will suppress the indices and simply write $\approx$ instead of $\approx_M^k$.
\end{definition}

\begin{proposition}\label{prop:bal_cond_implies_loc_const_stella}
    Let $M$ be a matroid on $E$ and $1\leq k\leq\crk(M)$. Then, any Minkowski weight from $\MW^k(\Sigma_M^k)$ is constant on $\approx$-equivalence classes.
\begin{proof}
    The balancing condition for $\Delta\in\MW^k(\Sigma_M^k)$ at a compatible pair $I\leq\mcF\in\Sigma_M^k(1)$ with $\mcF$ empty or $\mcF/F_0\in\uSigma_{M/F_0}^{k+|I|-|F_0|}(0)$ states that
    \begin{align*}
        \sum_{(J\leq\mcG)\succneq (I\leq\mcF)}\Delta(J\leq\mcG)(\mathbf{e}_{J\sm I}-\mathbf{e}_{E\sm(\mcG\sm\mcF)})\in\Span(\sigma_{I\leq\mcF}).
    \end{align*}
    If $F_0\subset\cl_M(I)$, this expression becomes
    \begin{align*}
        \Bigg(\sum_{i\in F_0\sm I}\bigl(\Delta(I\cup\{i\}\leq\mcF)-\Delta(I\leq(F_0\sm\{i\}\subsetneq\mcF))\bigr)\mathbf{e}_i\Bigg) + \text{terms supported in $E\sm F_0$}.
    \end{align*}
    Since $\Span(\sigma_{I\leq\mcF})$ is spanned by vectors with support in $I\cup(E\sm F_0)$, the left summand has to be zero. Thus, we obtain
    \begin{align*}
        \Delta(I\cup\{i\}\leq\mcF)=\Delta(I\leq(F_0\sm\{i\}\subsetneq\mcF)) \quad\text{for all $i\in F_0\sm I$}.
    \end{align*}
    If $F_0\not\subset\cl_M(I)$, the balancing condition implies, using $\mathbf{e}_{E\sm F_0}\in\Span(\sigma_{I\leq\mcF})$, that
    \begin{align*}
        \Bigg(\sum_{i\in F_0\sm\cl_M(I)}\Delta(I\cup\{i\}\leq\mcF)\mathbf{e}_i\Bigg)-\Delta(I\leq(\cl_{M|F_0}(I)\subsetneq\mcF))\mathbf{e}_{F_0\sm\cl_M(I)} \in \Span(\sigma_{I\leq\mcF}).
    \end{align*}
    By the same argument as before, we find that the above term is zero. In view of our definition of the relation~$\approx$, this completes the proof.
\end{proof}
\end{proposition}

\begin{lemma} \label{lem:connectedness_approx}
    Let $M$ be a matroid on $E$ and suppose $1\leq k\leq\crk(M)$. Then every compatible pair $J\leq\mcG$ from $\Sigma_M^k(0)$ is $\approx$-equivalent to another compatible pair from $\Sigma_M^k(0)$ of the form $\emptyset\leq(\emptyset\subsetneq\mcG')$.
\begin{proof}
    Note that there is nothing to prove if $J$ is empty, as this implies $G_0\subset\cl_M(\emptyset)=\emptyset$ using our standing assumption of $M$ being loopless. Thus, let $J$ be non-empty and pick $e\in J$. By applying Lemma~\ref{lem:S_M^k_and_E_M^k} to $J\sm\{x\}\leq\mcG$, we find in either of the two cases some $S\subsetneq G_0$ such that $J\sm\{e\}\leq(S\subsetneq\mcG)\in\Sigma_M^k(0)$ is $\approx$-equivalent to $J\leq\mcG$. Proceeding inductively yields the claim.
\end{proof}
\end{lemma}

\begin{lemma} \label{lem:perm_is_star_in_stella}
    Let $M$ be a matroid on $E$ and $1\leq k\leq\crk(M)$. We have
    \begin{align*}
        \mcG\in\uSigma_M^k(0) \iff\emptyset\leq(\emptyset\subsetneq\mcG)\in\Sigma_M^k(0).
    \end{align*}
    Hence, if $\operatorname{star}_{\emptyset\leq(\emptyset)}(\Sigma_M^k)$ is the subfan of $\Sigma_M^k$ consisting of cones containing the ray $\mathbf{e}_E$ corresponding to $\emptyset\leq(\emptyset)$, then one can recover $\uSigma_M^k$ from it by quotienting by $\Span(\mathbf{e}_E)$.
\begin{proof}
    This is immediate from \eqref{eq:Sigma_M^k(n-k)} and \eqref{eq:Sigma_M^k(n-k)_stella}.
\end{proof}
\end{lemma}

\begin{proof}[Proof of Theorem~\ref{thm:main_result_stella}]
    Note that $\torstar_{\emptyset\leq(\emptyset)}(\Sigma_E)$ is a subfan of $\Sigma_E$ and the inclusion induces a surjective pullback map $\MW^\bullet(\Sigma_E)\to\MW^\bullet(\torstar_{\emptyset\leq(\emptyset)}(\Sigma_E))$ that sends a Minkowski weight on $\Sigma_E$ to its restriction to $\torstar_{\emptyset\leq(\emptyset)}(\Sigma_E)$. By Lemma~\ref{lem:perm_is_star_in_stella}, the space $\MW^k(\torstar_{\emptyset\leq(\emptyset)}(\Sigma_E))$ is isomorphic to $\MW^k(\uSigma_E)$ and under this identification the above pullback sends weights supported in $\Sigma_M^k(0)$ to weights supported in $\uSigma_M^k(0)$. Thus, we have the following commutative diagram:
    \begin{figure}[h]
        \centering
        \[\begin{tikzcd}
        	{\MW^k(\Sigma_E)} & {\MW^k(\torstar_{\emptyset\leq(\emptyset)}(\Sigma_E))} &
            {\MW^k(\uSigma_E)} & {} \\
        	{\MW^k(\Sigma_M^k)} && {\MW^k(\uSigma_M^k)}
        	\arrow[two heads, from=1-1, to=1-2]
            \arrow["{\sim}", from=1-2, to=1-3]
        	\arrow[hook, from=2-1, to=1-1]
        	\arrow[from=2-1, to=2-3]
        	\arrow[hook, from=2-3, to=1-3]
        \end{tikzcd}\]
    \end{figure}\\
    The combination of Proposition~\ref{prop:bal_cond_implies_loc_const_stella} and Lemma~\ref{lem:connectedness_approx} shows that the lower map is an isomorphism. In view of Proposition~\ref{prop:char_nef_rays}, this proves the equivalence of the conditions from Theorem~\ref{thm:main_result} and the extremality of $c_k(\mcQ_M)$ in $\Nef^k(X_E)$.
\end{proof}

\begin{corollary}
    Let $M$ be a loopless matroid on $E$. Then its augmented Bergman class $\Delta_M\cap[X_E]$ generates an extremal ray of $\Nef^{\crk(M)}(X_E)$.
\begin{proof}
    By \cite[Cor.~5.11]{Eur_Huh_Larson_2023}, the augmented Bergman class of $M$ is equal to $c_{\crk(M)}(\mcQ_M)$. Using our notation, we thus have $\Delta_M=\Delta_M^{\crk(M)}$. Hence, the claim follows from the extremality of the non-augmented Bergman class of $M$ by the exact same argument we used in the previous proof.
\end{proof}
\end{corollary}

\begin{lemma}
     Let $f:X_{E\cup\{e\}}\to X_{E}$ be the toric morphism induced from the projection $\RR^{E\cup\{e\}}\to\RR^E$. Then, the subspace $f^*\CH^k(X_E)$ of $\CH^k(X_{E\cup\{e\}})$ is the orthogonal complement with respect to the degree pairing of the space spanned by $x_{I\leq\mcF}$, where $I\leq\mcF$ ranges over those compatible pairs for which $e\notin I$ and $e$ is contained in a gap of $\mcF$.
\begin{proof}
    We can argue as in Lemma~\ref{lem:pullback_subspace}: Since the dimension of $V(\sigma_{I\leq\mcF})$ drops after taking the image under $f$ for every $I\leq\mcF$ as in the claim, we see that every element of $f^*\CH^k(X_E)$ is indeed orthogonal to $x_{I\leq\mcF}$.\\
    For the other direction, suppose that $y\in\CH^k(X_{E\cup\{e\}})$ is orthogonal to all $x_{I\leq\mcF}$ as in the claim.
    In terms of Minkowski weights, this precisely means that no pair $I\leq\mcF$ with $e\notin I$ and $\mcF$ having a gap containing $e$ is in the support of $\Delta=y\cap\Delta_{\Sigma_E}$. Suppose that $J\leq\mcG\in\Sigma_\Delta(1)$ with $n\notin J$ and $n\in G_0$. The balancing condition at $J\leq\mcG$ states that
    \begin{align*}
        \Delta\bigl(J\cup\{e\}\leq\mcG\bigr)\mathbf{e}_e - \Delta\bigl(J\leq (G_0\sm\{e\}\subsetneq\mcG)\bigr)\mathbf{e}_{(E\sm G_0)\cup\{e\}}
        \in\Span(\sigma_\mcG).
    \end{align*}
    This forces the two appearing values of $\Delta$ to be equal. In view of Lemma~\ref{lem:pullback_subspace}, we obtain that, for $I\leq\mcF,I'\leq\mcF'\in\Sigma_\Delta(0)$ with $(I\sm\{e\}\leq\mcF\sm e)=(I'\sm\{e\}\leq\mcF'\sm e)$, the two values $\Delta(I\leq\mcF)$ and $\Delta(I'\leq\mcF')$ are equal. Thus, $\Delta$ is the pullback of the Minkowski weight on $\Sigma_E$ that assigns to a pair $K\leq\mcH$ the value of $\Delta$ at any pair $I\leq\mcF$ with $(I\sm\{e\}\leq\mcF\sm e)=(J\leq\mcH)$.
\end{proof}
\end{lemma}

Precisely as in Corollary~\ref{cor:lifting_property_of_nef}, this implies:
\begin{corollary} \label{cor:lifting_property_of_nef_stella}
    Let $E'\subsetneq E$ and $f:X_E\to X_{E'}$ be the toric morphism induced from the projection $\RR^E\to\RR^{E'}$. For $x\in\CH^k(X_{E'})$, we have that $x$ is nef if and only if the pullback $f^*x$ is nef. The cone $f^*(\Nef^k(X_{E'}))$ is a face of $\Nef^k(X_E)$ and $x$ generates an extremal ray of $\Nef^k(X_{E'})$ if and only if $f^*x$ generates an extremal ray of $\Nef^k(X_E)$.
\end{corollary}

\begin{remark}
    In contrast to the non-augmented case, if $L$ consists of a loop $l$, then the class $c_k(\mcQ_{M\oplus L})$ cannot be extremal in $\Nef^k(X_E)$ for any matroid $M$ on $E$: If $f:X_{E\cup\{l\}}\to X_E$ is the toric morphism induced from the projection of lattices, it is clear from \cite[Prop.~4.4]{Eur_Huh_Larson_2023} that there holds 
    \begin{align*}
        [\mcQ_{M\oplus L}]= f^*[\mcQ_M]+[\mathcal{L}]
    \end{align*}
    in the equivariant $K$-ring $K_T^0(X_{E\cup\{l\}})$ where $[\mathcal L]_{I\leq\mcF}$ is $1$ if if $l\in I$ and $T_l^{-1}$ otherwise. Since the piecewise linear function corresponding to $[\mathcal{L}]$, which is $0$ on $\sigma_{I\leq\mcF}$ if $l\in I$ and $-t_l$ otherwise, is convex, the decomposition 
    \begin{align*}
        c_k(\mcQ_{M\oplus L})= f^*c_k(\mcQ_M)+f^*c_{k-1}(\mcQ_M)c_1(\mathcal L).
    \end{align*}
    shows that $c_k(\mcQ_{M\oplus L})$ is not extremal.
\end{remark}

\section{Future directions}
By the fan displacement rule from \cite{Fulton_Sturmfels_1997}, it is clear that products of non-negative Minkowski weights on complete fans are again non-negative. Hence, products of nef Chow classes on complete toric varieties are again nef. While the product of extremal nef classes need not be extremal (see \cite[Exp.~5.9]{Huh_2016} for an example on $\uX_{[3]}$), computer experiments suggest that, for $n$ small, many extremal rays arise as products of tautological classes corresponding to possibly different matroids. This leads to the following question:
\begin{question}
    What are necessary and sufficient conditions on matroids $M_1,\dots,M_s$ and integers $r\leq s$,  $k_1,\dots,k_s$ for the product
    \begin{align*}
        c_{k_1}(\umcQ_{M_1})\cdots c_{k_r}(\umcQ_{M_r})c_{k_{r+1}}(\umcS_{M_{r+1}}^\vee)\cdots c_{k_s}(\umcS_{M_s}^\vee)
    \end{align*}
     to generate an extremal ray of the nef cone of the permutohedral variety?
\end{question}
A known instance of this question is the case where the factors are tautological top Chern classes. By \cite[Prop.~4.4]{Speyer_2008}, one has for loopless matroids $M$ and $M'$ that the product of their Bergman classes $c_{\crk(M)}(\umcQ_M)c_{\crk(M')}(\umcQ_{M'})$ is the Bergman class $c_{\crk(M)+\crk(M')}(M\wedge M')$, where $M\wedge M'$ is the matroid intersection of $M$ and $M'$. Thus, this construction cannot yield any new extremal rays.\\
If one specializes the above question to the case of products of tautological classes of a single matroid, extremality becomes quite rare. This is because such products will often decompose as a sum of \textit{Schur classes}, which are evaluations of Schur polynomials indexed by Young diagrams in tautological Chern classes, see \cite{Berget_Fink_2021} and \cite{Berget_Eur_Spink_Tseng_2023}. In the representable case, such Schur classes are known to be nef as shown in~\cite{Fulton_Lazarsfeld_1983}. In the non-representable case, their nefness can be seen to be equivalent to a conjecture from \cite{Berget_Fink_2021}, which is widely open. A natural follow-up question to this observation is thus:
\begin{question}
    Provided a certain Schur class of a matroid is nef, what can be said about its extremality in the nef cone of the permutohedral variety?
\end{question}
Of course, this idea can again further be generalized by asking the same question for products of Schur classes of matroids.
Considering such products as potential candidates for new extremal rays also seems quite natural in view of definitions made in \cite{Fulger_Lehmann_2014}, where the authors constructed a subcone of the nef cone of any projective variety $X$, called the \textit{pliant cone}. By definition, it is precisely spanned by products of Schur classes of globally generated vector bundles on $X$.

\bibliography{biblio}
\bibliographystyle{amsalpha}

\end{document}